\documentclass[12pt]{amsart}

\usepackage{latexsym, amssymb, amsmath,ulem,soul,esint}

\usepackage{amsfonts, graphicx}
\usepackage{graphicx,color}
\newcommand{\be}{\begin{equation}}
\newcommand{\ee}{\end{equation}}
\newcommand{\beq}{\begin{eqnarray}}
\newcommand{\eeq}{\end{eqnarray}}

\newtheorem{thm}{Theorem}[section]

\newtheorem{lma}{Lemma}[section]
\newtheorem{prop}{Proposition}[section]
\newtheorem{cor}{Corollary}[section]

\theoremstyle{remark}
\newtheorem{rem}{Remark}[section]
\numberwithin{equation}{section}
\def\tr{\operatorname{tr}}

\def\be{\begin{equation}}
\def\ee{\end{equation}}
\def\bee{\begin{equation*}}
\def\eee{\end{equation*}}

\def\lf{\left}
\def\ri{\right}

\def\div{\mathbf{div}}

\def\m{\mathfrak{m}}

\def\Ric{\text{\rm Ric}}

\def\cS{\mathcal{S}}

\def\wn{\wt\nabla}

\def\wt{\widetilde}
\def\la{\langle}
\def\ra{\rangle}
\def\p{\partial}

\def\tr{\operatorname{tr}}

\def\e{\varepsilon}
\def\a{{\alpha}}
\def\b{{\beta}}

\def\R{\mathbb{R}}

\def\Vol{\mathrm{Vol}}

\def\cA{\mathcal{A}}
\def\cB{\mathcal{B}}

\def\cG{\mathcal{G}}

\begin{document}

\title{Monotone quantities of $p$-harmonic functions and their applications}

\author{Sven Hirsch}
\address{Department of Mathematics \\ Duke University \\ Durham, NC, 27708}
\email{sven.hirsch@duke.edu}

\author{Pengzi Miao$^\dagger$}
\address{Department of Mathematics \\ University of Miami\\ Coral Gables, FL, 33146}
\email{pengzim@math.miami.edu}

\author{Luen-Fai Tam$^*$}
\address{The Institute of Mathematical Sciences and the Department of Mathematics \\ The Chinese University of Hong Kong \\ Shatin, Hong Kong, China}
\email{lftam@math.cuhk.edu.hk}

\thanks{$^\dagger$Research partially supported by NSF grant DMS-1906423.}

\thanks{$^*$Research partially supported by Hong Kong RGC General Research Fund \#CUHK
14301517}

\maketitle

\markboth{Sven Hirsch, Pengzi Miao and Luen-Fai Tam}{$p$-harmonic functions and the mass}

\begin{abstract}
We derive local and global monotonic quantities associated to $p$-harmonic functions on manifolds with nonnegative scalar curvature.
As applications, we obtain inequalities relating the mass of asymptotically flat $3$-manifolds,
the $p$-capacity and the Willmore functional of the boundary.
As $ p \to 1$, one of the results retrieves a classic relation
that the ADM mass dominates the Hawking mass if the surface is area outer-minimizing.
\end{abstract}

\section{Introduction and statement of results}\label{s-intro}

The Riemannian Penrose inequality (RPI) in $3$-dimension relates the mass of an asymptotically flat
manifold to the area of its boundary if the boundary is the outermost
minimal surface in the sense that it is not
enclosed by another minimal surface.  The inequality was proved by
Huisken-Ilmanen \cite{HuiskenIlmanen2001}
in the case of connected boundary via a weak formulation of the inverse mean curvature flow.
The general case was proved by Bray \cite{Bray02} using a conformal flow of metrics. The higher
dimensional inequality was proved by Bray and Lee \cite{BrayLee09} in dimensions less than eight.

Recently, Agostiniani-Mantegazza-Mazzieri-Oronzio \cite{AMMO2022} gave
a new proof of the $3$-dimensional RPI in the case of connected boundary.
In their approach, a related inequality was  first established via $p$-harmonic functions with $p>1$. The  RPI was obtained  by letting $p \to 1$ using \cite[Theorem 1.2]{FM22} on the limiting behavior of $p$-capacity.
A main ingredient in the approach was a generalization to $p$-harmonic functions
of a monotonic property of harmonic functions found by Agostiniani-Mazzieri-Oronzio in \cite{AMO2021}.

In \cite{Miao2022}, some distinct monotone properties of harmonic functions were found by the second author.
It is natural to ask if those properties of harmonic functions can be generalized to $p$-harmonic functions,
and if so, what applications such a generalization may give.

This paper is motivated by the above questions.
Among other things, we retrieve in Corollary \ref{cor-mass-pc-Willmore}
that
\be \label{eq-g-Hawkingmass-m-intro}
\sqrt{ \frac{ | \p M | }{ 16 \pi } }  \left(
1 - \frac{1}{16 \pi} \int_{\p M} H^2 \right) \le   \mathfrak{m}, \
\ \text{if } \p M \ \text{is area outer-minimizing} .
\ee
Here area outer-minimizing means that every surface $\Sigma$ enclosing $\p M$ has larger area.
\eqref{eq-g-Hawkingmass-m-intro} represents a well-known relation
$$ \text{Hawking mass} \le \text{ADM mass} , $$
provided $\p M$ is area outer-minimizing
and $M$ has simple topology.
This was first proved by Huisken and Ilmanen  \cite{HuiskenIlmanen2001} via the method of
weak inverse mean curvature flow. In Corollary \ref{cor-mass-pc-Willmore}, we show \eqref{eq-g-Hawkingmass-m-intro}
can be derived from $p$-harmonic functions.

We recall that a complete Riemannian $3$-manifold $(M,g)$
is asymptotically flat (AF) if there is a compact set $K$ such that $M\setminus K$ is diffeomorphic to the exterior of a ball in
$\R^3 $, such that if $\delta_{ij}$ is the Euclidean metric, then for $0\le l\le 2$
\be\label{e-AF-1}
| g_{ij}-\delta_{ij}|=O_2(r^{-\sigma})
\ee
as $r\to\infty$,  with $\sigma>\frac {1}2$, where $r$ is the Euclidean distance from a fixed point. This means
$|D^l(g_{ij}-\delta_{ij})|=O(r^{-l-\sigma})$ for $0\le l\le 2$. Suppose the scalar curvature $\cS$ of $M$ is integrable, then the ADM mass $\mathfrak{m}$ introduced in \cite{ADM61} is well-defined \cite{Bartnik,Chrusciel}:

\be\label{e-ADMmass}
\mathfrak{m}= \frac{1}{16 \pi} \lim_{r\to\infty}\int_{ S_r }\lf(g_{ij,i}-g_{ii,j}\ri)\nu^j_ed\sigma_e ,
\ee
where $S_r=\{|x|=r\}$,  $\nu_e$ is the unit outward normal to $S_r$ and $d\sigma_e$ is the area element on $S_r$  both with respect to the Euclidean metric $g_e$.
In this work, we prove the following:

\begin{thm}\label{thm-ratio-intro}
Let $(M^3,g)$ be a complete, orientable, asymptotically flat $3$-manifold with boundary $\p M$.
Suppose $\p M$ is connected and $H_2(M,\p M)= 0$.
For $1<p\le 2$, let $u$ be the $p$-harmonic function on $  M $ with $u=0$ on $\p M $, and $u\to 1$ at infinity.
If $g$ has nonnegative scalar curvature, then
\be \label{ineq-hu-gdu-intro}
4 \pi  +  \int_{\p M }   | \nabla u | H \ge a^{-2} ( 1 + 2a) \int_{\p M} | \nabla u |^2   ,
\ee
\be \label{ineq-Hup-intro}
c^{\frac1a} \left( 8\pi- a^{-1} \int_{\p M}|\nabla u|H \right) \le 4\pi (5-p)\mathfrak{m},
\ee
\be \label{ineq-dup-intro}
c^{\frac1a} \left(  4\pi-a^{-2} \int_{\p M}|\nabla u|^2 \right)
\le 4\pi (3-p)\mathfrak{m}.
\ee
Here $a = \frac{3-p}{p-1}$,
$ c  = a^{-1}\left( \frac{C_p}{ 4\pi} \right)^\frac{1}{p-1} $,  $C_p$ is the $p$-capacity of $\p M$ in $(M,g)$
and $H$ is the mean curvature of $\p M $.
If equality holds in any of the above inequalities, then $(M,g)$ is isometric to $\R^3$ outside a round ball.
\end{thm}
 For simplicity, we only consider orientable manifolds in this work. By taking two-fold cover, some results are also true for non-orientable manifolds.
Besides retrieving \eqref{eq-g-Hawkingmass-m-intro}, other applications of Theorem \ref{thm-ratio-intro} include
sufficient conditions via $C^0$-data of regions separating the boundary and
the infinity, which imply  the positivity of the mass.
See Corollary \ref{thm-sec-pmt-bdry} for details.

We prove Theorem \ref{thm-ratio-intro} by exhibiting
a family of monotonic quantities for $p$-harmonic functions
and analyzing their asymptotic behavior.
 
First, let us  review the concepts of $p$-harmonic functions and $p$-capacity, {also see Section 1.1 in \cite{AMMO2022}}.
Let $(M, g)$ denote a complete Riemannian $3$-manifold  with boundary $\p M$, which is assumed to be compact throughout this work.
Given any $ p \in (1, 3)$, a function $u$ is called a $p$-harmonic function such that
$ u $ vanishes at $ \p M$ and $ u \to 1$ at infinity if  $u\in W^{1,p}_{loc} (M)$ and  satisfies
 \be\label{e-p-harmonic}
\left\{
  \begin{array}{ll}
    \div(|\nabla u|^{p-2}\nabla u)=0, & \hbox{in $M$ in the weak sense};\\
u=0&\hbox{on  $\p M$};\\
u(x)\to 1 &\hbox{as $x \to\infty$.}
  \end{array}
\right.
\ee
In case of asymptotically flat manifold, such a $u$ exists in $W^{1,p}_{loc}$ and
on any precompact set  $u$ is in $C^{1,\beta} $  for some $\beta>0$.
Moreover, $u$ is smooth whenever $|\nabla u|>0$, see \cite[Theorems 1,2]{Benedetto1}, see also \cite{Benedetto2}.
 the maximum principle holds for $u$, see \cite[{Lemma 3.18 and Theorem 6.5}]{HKM93},
which implies $ 0 \le u < 1$ and $ \p M = \{ u = 0 \}$.
Moreover, the Hopf lemma holds, see \cite[{Section 2}]{T83}, which shows $ | \nabla u | > 0 $ at $ \p M$.
As $x \to \infty$, $u$ has an asymptotic expansion of
\be\label{e-behavior}
u=1-cr^{-a}+o_2(r^{-a}) , \ \ r = | x|.
\ee
That is $|D^\ell u-D^\ell(1-cr^{-a})|=o(r^{-a-\ell})$ for $\ell=1, 2$.
See \cite[Theorem 3.1]{BenattiFogagnoloMazzieri}, where
$a=\frac{3-p}{p-1}$,   $ c  = a^{-1} \left(\frac{C_p}{4\pi} \right)^\frac 1{p-1} $,
and $C_p$  is the $p$-capacity of $\p M $ in $(M, g)$ given by
\bee
C_p = \inf \left\{ \int_M | \nabla \phi |^p \right\} ,
\eee
where the infimum is taken over all Lipschitz functions $\phi$ with compact support such that $\phi = 1$ at $ \p M$.
$C_p$ is related to $u$ by
$$ C_p = \int_M | \nabla u |^p = \int_{\{u = \tau\}} | \nabla u |^{p-1}  $$
if $\tau $ is a regular value of $u$, see \cite{BFM} for instance. Here we omit the volume element and the area element for simplicity.

We want to study quantities related to the mass of an AF manifold.
It was know in \cite[Lemma 2.2]{FanShiTam2009} (also see \cite[Proposition A.2]{JangMiao21}) that, as $ r \to \infty$,
\be \label{eq-H-A-m}
4\pi r-\int_{S_r}H +\frac{A(r)}r=8\pi \m  +o(1).
\ee
Here $H$ is the mean curvature and $A(r)$ is the area of $S_r$, respectively.
From this, one can check that any function $f(t)$ with $ f'(t) > 0 $, and if $ v (x) = f ( r (x) )  $, then
 \be \label{eq-vr-H-m}
 \begin{split}
 & \  4 \pi r - f'(r)^{-1} \int_{ \{v = f(r)\}  }  | \nabla v | H  +   f'(r)^{-2} r^{-1}\int_{ \{v  = f(r)\} }    | \nabla v |^2 \\
 = & \ 8\pi \m +o(1) .
 \end{split}
 \ee
  Given the $p$-harmonic function $u$, motivated by \eqref{e-behavior}, let   $f (t)=1-ct^{-a}$, if $f(t)$ is a regular value of $u$, then the above expression becomes:
\be \label{eq-F}
F (t) =4\pi t -(ca)^{-1}t^{a+1}\int_{\{ u = f(t)\}  }|\nabla u| H  + (ca)^{-2}t^{2a+1}\int_{\{ u = f(t)\} }|\nabla u|^2 ,
\ee
which is the quantity considered in  \cite{AMO2021, AMMO2022}.

Motivated by  \cite{Miao2022}, two other quantities may be constructed from
$u$ and $f$.
We define
\be \label{e-AB-1}
\left\{
  \begin{array}{ll}
    \cB(t)=4\pi t-(f')^{-2}t^{-1}\int_{\{ u = f(t)\}   }|\nabla u|^2=4\pi t-(ca)^{-2}t^{2a+1}\int_{ \{u = f(t)\} }|\nabla u|^2, \\
     \ \\
    \cA(t)=8\pi t-(f')^{-1}\int_{\{ u = f(t)\}  }|\nabla u| H=8\pi t-(ca)^{-1}t^{a+1}\int_{ \{u = f(t)\}  }|\nabla u| H.
  \end{array}
\right.
\ee
$F(t), \cA(t), \cB(t)$ are related to the Hawking mass of the level surface, see Appendix C.

Similar to $\Psi( s  )$ in \cite[Section 3]{Miao2022}, we also define
\be\label{e-D}
\begin{split}
D (t)= & \ t^{-a} \cB'(t) \\
= & \  4\pi t^{-a}  + c^{-1}   \int_{\{ u = f(t)\} }|\nabla u|H - (ca)^{-2}(1+2a)t^{a}\int_{\{u=f(t)\}}|\nabla u|^2 .
\end{split}
\ee
On the other hand,
if $(M, g)$ is a complete manifold without boundary and
$G$ is the positive $p$-harmonic Green's function with pole at $x_0$ and approaching $0$ at infinity   (provided it exists), then define
\be\label{e-pG-intro}
\cG(\tau)=-4a^2\pi \tau+\tau^{-1}\int_{ \{G = \tau\} }|\nabla G|^2.
\ee
This  was studied in \cite{MunteanuWang2021, CCLT2022}.  We will see in Section 2 that $\cA(t)$, $\cB(t)$, $D(t)$, $F(t)$ and $\cG (\tau)$ and their monotone properties are closely related.

The monotone property of $F$ for $p$-harmonic functions was proved in \cite{AMO2021,AMMO2022} and the monotone properties of $\cA, \cB, D$ for harmonic functions were obtained in \cite{Miao2022}. In this work, we prove the following:
\begin{thm}\label{t-intro-1}
For $1<p\le 2$, suppose $u$ is  a $p$-harmonic function on
a complete, orientable Riemannian $3$-manifold $(M, g)$ with compact boundary $\p M$
such that  $u=0$ on $\p M $ and $u\to 1$ at infinity.
Suppose $ \p M$ is connected, $H_2 (M, \p M) = 0 $, and $g$ has non-negative scalar curvature.
Let $f(t), \cA(t), \cB(t), D(t)$ be given as above. Let $\Sigma(t)=\{u=f(t)\}$.
Suppose $ 0 < t_1<t_2$
 such that $\Sigma(t_1) $ and $\Sigma(t_2)$ are regular.
Then
   \begin{enumerate}

     \vspace{.2cm}

     \item [(i)]{  (Local monotonicity) $D(t_1)\ge D(t_2)$.
          Moreover, if $D(t_1)=D(t_2)$, then   $\{f(t_1)<u<f(t_2)\}$ is isometric to an annulus in $\R^3$. }

     \vspace{.2cm}

       \item[(ii)] (Global positivity) If $(M, g)$ is asymptotically flat,
       then $D(t)\ge0$ and   $(1+2a) \cB(t) - a \cA(t)\ge 0$ whenever $\Sigma(t)$ is regular.
       Moreover, equality holds if and only if $\{u>f(t)\}$ is isometric to $\R^3$ outside a round ball.

       \vspace{.2cm}

     \item [(iii)] (Global monotonicity) If $(M, g)$ is asymptotically flat, then
     $$ \cB(t_2)\ge \cB(t_1) \ \ \text{and} \ \ \cA(t_2)\ge \cA(t_1). $$
     {Moreover, if $ \cB(t_1)= \cB(t_2)$ or $ \cA(t_1)= \cA(t_2)$, then
     $\{u>f(t_1)\}$ is isometric to $\R^3$ outside a round ball.}
   \end{enumerate}
\end{thm}

Theorem \ref{t-intro-1}, together with the asymptotically behaviors of $\cA(t)$ and $\cB(t)$, i.e. Proposition \ref{t-asy}, will imply  Theorem \ref{thm-ratio-intro}.

Next,   we want to give a unified treatment on the   monotone properties regarding $\cA(t)$, $\cB(t)$, $D(t)$, $F(t)$ and $\cG (\tau) $ at least if no critical points are present. First we summarise the results we want to consider in the following table:
\begin{table}[h]
{\begin{tabular}{|c|c|c|c|c|}
    \hline
  &type&  $p=2$ & $p\in(1,2)$ & IMCF $(p=1)$ \\
    \hline
$F$   &  local & Agostiani &Agostiani & \\
 & & -Mazzieri & -Mantegazza&  \\
 && -Oronzio \cite{AMO2021}  & -Mazzieri-Oronzio \cite{AMMO2022}&Geroch \cite{Geroch1973}\\
    \cline{1-4}
$ D$ &  local & Miao & Theorem \ref{t-intro-1} &Jang-Wald \cite{JangWald1977} \\
$\cA, \cB$ & non-local & \cite{Miao2022}  & in this work& Huisken-Ilmanen \cite{HuiskenIlmanen2001} \\
 \cline{1-4}
$\mathcal G$ &non-local& Munteanu&Chan-Chu& \\
&&-Wang \cite{MunteanuWang2021}&-Lee-Tsang \cite{CCLT2022}&\\
\hline
\end{tabular}}
\label{table-example}
\end{table}

Here non-locality in this table means either one made use of  the asymptotically flatness of the manifold,
or one used the asymptotically behavior near the pole of the Green's function.
The above monotone properties are consequences of a single formula described in the following theorem, assuming there is no  critical point.

\begin{thm}\label{T: integral formula-h-intro}
Let $(M,g)$ be a compact, $3$-dimensional Riemannian manifold with boundary $\p M$.
Suppose $\partial M $ consists of two connected components $\partial_+M$, $\partial_-M$.
Let $\a\in [-1,1]$ and $\b=0$  or $\frac2{1-\a}$ if $|\alpha| < 1$. Suppose  $u$ is a solution with $|\nabla u|>0$ to the boundary value problem
\be \label{eq-bdry-system}
\left\{
  \begin{array}{ll}
    \Delta u=\a\nabla^2u(\nu,\nu)+\frac{2|\nabla u|^2}u, & \hbox{in $M$} \\
    u=c_+, & \hbox{at $\p M_+$}\\
u=c_-, & \hbox{at $\p M_-$}\\
  \end{array}
\right.
\ee
for two positive constants $c_-<c_+$.
Here $\nabla^2u$ denotes the Hessian of $u$ and $\nu=\frac{\nabla u}{|\nabla u|}$.
Then the following equality holds:
    \begin{align}\label{main formula}
    \begin{split}
   &\frac1{u^\b}\lf(\mathcal R_\a(u)- \cS^t|\nabla u| \ri)\\\
   =& \ 2\div\left[\frac1{u^\b}\left( \nabla |\nabla u|- \Delta u \frac{\nabla u}{|\nabla u|}+ \frac{2\b-1}{\b-1}\frac{|\nabla u|\nabla u}u\right)\right]=:\div\,(u^{-\b} X) ,
\end{split}
\end{align}
 where $\cS^t=2K$ denotes the scalar  curvature of the level sets $\Sigma_t=\{u=t\}$, where $K$ is the Gaussian curvature,
  and
\bee
\begin{split}
    \mathcal R_\a(u)= & \, \cS|\nabla u|+|\nabla u|^{-1}|T|^2 -\a^2|\nabla u|^{-1} u_{\nu\nu}^2\\
\end{split}
\eee
in which $\cS$ is the scalar curvature of $M$, $ u_{\nu \nu} = \nabla^2 u (\nu, \nu)$, and
$$
T=\nabla^2u+u^{-1}\lf(\nabla u\otimes \nabla u - |\nabla u|^2 g\ri).
$$
As a result, upon integration,
\be \label{eq-integration-mono}
\begin{split}
   \frac12& \int_M \frac1{u^\b}  \mathcal R_\a(u) dV
   \\ =&2\pi\chi\int_{c_-}^{c_+}\frac1{t^\b}dt+    \lf(\int_{\partial_+ M}-\int_{\partial_- M}\ri)\frac1{u^b}\left(-H|\nabla u|+\frac{2\b-1}{\b-1}\frac{|\nabla u|^2}u\right)dA,
\end{split}
\ee
where $H=H_u$ is the mean curvature of the boundary with respect to $ \nabla u / | \nabla u | $
and $\chi$ is the Euler characteristic of $\p_-M$ and hence of every level set of $u$.
\end{thm}

A motivation to \eqref{eq-bdry-system} is that it gives the equation $|x|$ satisfies in the setting of $\R^3$,
where $|x|^{-a}$ is a $p$-harmonic function
and $ 2 \log |x|$ is a solution to the inverse mean curvature flow.
See the discussions following Corollary \ref{c-integral formula-h}.


We want to add that all monotonic properties of $p$-harmonic functions
mentioned above have a model space of the exterior of a round ball in $\R^3$.
In \cite{Oronzio2022}, Oronzio obtains monotonicity formulae modeled
on Schwarzschild manifolds.

The organization of the paper is as follows. In section 2, we will prove Theorem \ref{T: integral formula-h-intro} and study the relation between $\cA, \cB, D, F, \cG$. In section 3, we will prove Theorem \ref{t-intro-1}, and in section 4, we will study the asymptotical behavior of $\cA, \cB$ and prove Theorem \ref{thm-ratio-intro}. We will give applications in section 5 and list some facts in the appendices.

We would like to acknowledge that some results in the paper were known to Man-Chun Lee and Tin-Yau Tsang.
We thank for their communications \cite{LY-1}.
{
We are also grateful for the anonymous referees whose comments improved the paper.
}

\section{Relating $\cA$, $\cB$, $D$, $F$ and $\cG$}

We start this section with a proof of Theorem \ref{T: integral formula-h-intro}.

\begin{proof}[Proof of Theorem \ref{T: integral formula-h-intro}] Since $|\nabla u|>0$, we have $c_-<u<c_+$ in the interior of $M$.
The following computations are from   \cite[(4.8)]{BHKKZ2021} and \cite{Stern}.
By Bochner's identity and the Gauss equation,
\bee
\begin{split}
    2 \Delta|\nabla u | =& \ 2|\nabla u|^{-1}(|\nabla u|^2 \Ric(\nu,\nu)+|\nabla^2u|^2+\langle \nabla \Delta u,\nabla u\rangle-|\nabla |\nabla u||^2)\\
    =& \ 2|\nabla u|^{-1}( |\nabla^2u|^2+\langle \nabla \Delta u,\nabla u\rangle-|\nabla |\nabla u||^2)\\
    & \ +|\nabla u|\lf(\cS-\cS^t+H^2-|A|^2\ri)
\end{split}
\eee
where $H$ is the mean curvature and $A$ is the second fundamental form of the level sets $\Sigma_t$. By Lemma \ref{l-basic-1}:
\bee
|\nabla^2u|^2-2|\nabla |\nabla u||^2=|\nabla u|^2|A|^2-u_{\nu\nu}^2,
\eee
and
$$
H=|\nabla u|^{-1}(\Delta u-u_{\nu\nu}).
$$
Replacing $H^2-|A|^2$ with $u$ and its derivatives, one has
\bee
\begin{split}
    2 \Delta|\nabla u|= & |\nabla u|^{-1}( |\nabla^2u|^2+2\langle \nabla \Delta u,\nabla u\rangle+(\Delta u)^2-2u_{\nu\nu}\Delta u )+|\nabla u|\lf(\cS-\cS^t\ri).
\end{split}
\eee
As a result, one has the following formula, see  \cite[(4.8)]{BHKKZ2021}:
\bee
\begin{split}
2\div\lf(\nabla |\nabla u|-\Delta u\frac{\nabla u}{|\nabla u|}\ri)=&\frac{|\nabla^2u|^2}{|\nabla u|}+(\cS-\cS^t)|\nabla u|-\frac{(\Delta u)^2}{|\nabla u|}\\
=&\frac1{|\nabla u|}\lf(|\nabla^2u|^2-\a^2u_{\nu\nu}^2-4\a u_{\nu\nu}\frac{|\nabla u|^2}u-\frac{4|\nabla u|^4}{u^2}\ri)\\
&+(\cS-\cS^t)|\nabla u|.\\
\end{split}
\eee
Moreover,
\bee
\begin{split}
 \div\lf(\frac{|\nabla u|\nabla u}{u}\ri)=& \frac{ |\nabla u|\Delta u}{u}+\frac{\la\nabla|\nabla u|,\nabla u\ra}u-\frac{|\nabla u|^3}{u^2}\\
=&\frac{|\nabla u|}u\lf((\a+1)u_{\nu\nu}+\frac{ |\nabla u|^2}u\ri).
\end{split}
\eee
Since
\bee
\frac12X= \nabla |\nabla u|- \Delta u \frac{\nabla u}{|\nabla u|}+ \frac{2\b-1}{\b-1}\frac{|\nabla u|\nabla u}u ,
\eee
hence
\bee
\begin{split}
\la \nabla u^{-\b}, X \ra=&-2\b u^{-\b-1}\lf( \la \nabla|\nabla u|,\nabla u\ra-|\nabla u|\Delta u+\frac{2\b-1}{\b-1}\frac{|\nabla u|^3}u\ri)\\
=&-2\b u^{-\b-1}\lf( (1-\a)|\nabla u|u_{\nu\nu}  +\frac{1}{\b-1}\frac{|\nabla u|^3}u\ri).
\end{split}
\eee
Therefore,
\bee
\begin{split}
&u^{\b}\div\,(u^{-\b} X)-(\cS-\cS^t)|\nabla u|\\
=&  \frac1{|\nabla u|}\lf(|\nabla^2u|^2-\a^2u_{\nu\nu}^2-4\a u_{\nu\nu}\frac{|\nabla u|^2}u-\frac{4|\nabla u|^4}{u^2}\ri)\\
&+\frac{2(2\b-1)}{\b-1}\frac{|\nabla u|}u\lf((\a+1)u_{\nu\nu}+\frac{ |\nabla u|^2}u\ri) \\
&-2\b \frac{|\nabla u|}u\lf((1-\a)  u_{\nu\nu} +\frac{1}{\b-1}\frac{|\nabla u|^2}u\ri)\\
=&\frac{1}{|\nabla u|}\bigg[|\nabla^2u|^2-\a^2u_{\nu\nu}^2+2(1-\a) u_{\nu\nu}\frac{|\nabla u|^2}u-\frac{2|\nabla u|^4}{u^2}\bigg]\\
=&\frac{\left |\nabla^2u+u^{-1}\lf( \nabla u\otimes \nabla u - |\nabla u|^2 g\ri)\right|^2}{|\nabla u|}-\frac{\a^2u_{\nu\nu}^2}{|\nabla u|}.
\end{split}
\eee
Note that $T(\nu,\nu)=u_{\nu\nu}$ because $u_\nu=|\nabla u|$. This verifies \eqref{main formula}.

Upon integration, one has
\bee
\begin{split}
 & \int_M \frac1{u^\b}   \mathcal R_\a(u)   dV-\int_{c_-}^{c_+}\frac1{t^b}\lf(\int_{\Sigma_t}\cS^t\ri)dt
   \\ =&  \int_{\p M_+}\left\la   u^{-\b}X,\frac{\nabla u}{|\nabla u|}\right\ra dA -\int_{\p M_-}\left\la   u^{-\b}X,\frac{\nabla u}{|\nabla u|}\right\ra dA ,
\end{split}
\eee
as  the unit outward normal to $\p M$ is $\nu$ at $\p_+ M $ and $ - \nu $ at $\p_-M$.
Using the identity $|\nabla u|H_u=\Delta u-\frac{\la\nabla |\nabla u|,\nabla u\ra}{|\nabla u|}$, one obtains
\bee
\begin{split}
\left\la X,\frac{\nabla u}{|\nabla u|}\right\ra = & \, 2\lf(\frac{\la \nabla |\nabla u|,\nabla u\ra}{|\nabla u|}- \Delta u +\frac{2\b-1}{\b-1}\frac{|\nabla u|^2}u \ri)\\
= & \, 2\lf(-H_u|\nabla u| +\frac{2\b-1}{\b-1}\frac{|\nabla u|^2}u \ri)
\end{split}
\eee
which implies equation \eqref{eq-integration-mono} by the Gauss-Bonnet theorem.
\end{proof}
Since $\a^2\le 1$,
\bee
\mathcal R_\a(u)=   \, \cS|\nabla u|+|\nabla u|^{-1}(|T|^2-T^2(\nu,\nu))+(1 -\a^2)|\nabla u|^{-1} u_{\nu\nu}^2\ge \cS|\nabla u|
\eee
 we have:

 \begin{cor}\label{c-integral formula-h}
 Suppose $\cS\ge0$ in Theorem \ref{T: integral formula-h-intro}, then
\be \label{cor-ineq}
\begin{split}
     2 \pi \chi  & \int_{c_-}^{c_+}\frac1{t^\b}   dt
     + \lf(\int_{\partial_+ M}-\int_{\partial_- M}\ri)\frac1{t^\b}\left(-H_u|\nabla u|+\frac{2\b-1}{\b-1}\frac{|\nabla u|^2}u\right)dA \\
\ge & \ 0 ,
\end{split}
\ee
where $\chi$ is the Euler characteristic of $\p_-M$.
If the equality holds and $ \a <1$,
$(M^3,g)$   is isometric to an annulus in $\R^3$ and $u=C\rho$ where $\rho$ is the Euclidean distance to the center of the annulus.
\end{cor}

\begin{rem}
One can classify the rigidity case of $\a  = 1$ too. As it is not needed in the main results,
we include it in Appendix \ref{s-equations}.
\end{rem}

\begin{proof} \eqref{cor-ineq} follows from Theorem \ref{T: integral formula-h-intro}.
If equality holds, then  $\cS=0$ and
\be
|T|^2 - T^2 (\nu, \nu) + ( 1 - | \alpha |^2 ) u_{\nu\nu}^2=0 .
\ee
Let $v=u^2$, $t_0^2=c_-$, $t_1^2=c_2$ so that $t_0\le v\le t_1$, then
$ T=\frac1u\lf(\nabla^2v-\frac{|\nabla v|^2}{2v}g\ri) $.
Since $|\a|\le 1$, we have  $T(X,Y)=0, T(X,\nu)=0$, for any $X, Y$ tangent to the level set of $v$,
and
\be\label{e-rigidity-1}
\nabla^2v(X,Y)=|\nabla v|A(X,Y);\  \nabla^2v(X,\nu)=X(|\nabla v|) , \
\ee
where $\nu=\nabla v/|\nabla v|$. Hence $|\nabla v|=:\eta$ is constant on each level set, and on $\{v=t\}$,
$$
A(X,Y)=\frac{|\nabla v|}{2v}g(X,Y), \ \ H=\frac{ |\nabla v|}{ v}=\frac\eta t,
$$
Therefore, the level sets are umbilical and
$$
g=\eta^{-2}(t)dt^2+\gamma_t ,
$$
where $\eta(t)=|\nabla v|$ depends only on $t$ and $\gamma_t$ is the induced metric on $\{v=t\}$. Now $\p_t \gamma_t=2\eta^{-1}(t)A_t=t^{-1}\gamma_t$ where $A_t$ is the level set $\Sigma_t=\{v=t\}$. So
$  \gamma_t= tt_0^{-1}\gamma_{t_0} $.
It remains to find $\eta$.

If $ | \a |<1$, then $u_{\nu\nu}=0$. So
$$\eta'=\frac1{|\nabla v|}v_{\nu\nu}=\frac{2u_\nu^2}{|\nabla v|}=\frac{|\nabla v|}{2v}=\frac1{2t}\eta .
$$
If $\a=-1$, then
\bee
\begin{split}
\frac{\eta^2}t= |\nabla v|H = \Delta v-v_{\nu\nu}
= -2v_{\nu\nu}+\frac{2\eta^2}t
= -2\eta\eta'+\frac{2\eta^2}t.
\end{split}
\eee
We still have $ \eta'=\frac1{2t}\eta $.
Hence, in either case,  $t^{-\frac12}\eta(t)=t^{-\frac12}_0\eta(t_0)$. So if we let $t=r^2$,  then
\bee
\begin{split}
g= & \ t_0\eta^{-2}(t_0)t^{-1}dt^2+t_0^{-1}t\gamma_{t_0}\\
=& \ t_0^{-1}\lf((n-1)^2H^{-2}(t_0) dr^2+r^2 \gamma_{t_0}\ri).
\end{split}
\eee
To find $\gamma_{t_0}$,  by Lemma \ref{l-basic-1} and the facts that $|\nabla v|$ is constant on the level set and the level set is umbilical, we have
\bee
\begin{split}
\eta(t)\frac{\p }{\p t}H=&-\frac12\lf( \frac32H^2-2K\ri) .
\end{split}
\eee
On the other hand
\bee
\frac{\p}{\p t}H=\frac{\p}{\p t}\lf(\frac{ \eta}{ t}\ri)=-\frac{ \eta}{2t^2}.
\eee
So $K=\frac14 H^2$ which is a positive constant. Thus, each level set is a sphere, and
$$
g=4t_0^{-1}H^{-2}(t_0)\lf(dr^2+r^2\sigma_0\ri),
$$
where $\sigma_0$ is the standard unit sphere in $\R^3$. This completes the proof.
\end{proof}

Now we illustrate how Corollary \ref{c-integral formula-h} relates to the previously mentioned {\it local} monotone quantities.
It should be emphasized that Corollary  \ref{c-integral formula-h} assumed $|\nabla u|>0$.
Such an assumption was not necessary in the corresponding results below.

\vspace{.2cm}

{\bf(I)} {\sl  Inverse mean curvature flow}:
Take $ \a=1$ and $ \b=0$.
The function $U=2\log u$ satisfies $\Delta U=\nabla_{\nu\nu}U+|\nabla U|^2$.
This means the level sets $\{\Sigma_U\}$ of $U$ flow by inverse mean curvature flow
and $|\Sigma_U|=|\Sigma_{0}|e^U$.
In this case,
\begin{align*}
0\le& \ 4\pi\int_{c_-}^{c_+} 1 \, dt
     +\lf(\int_{\partial_+ M}-\int_{\partial_- M}\ri)  \left(-2H|\nabla u|
     +2 \frac{|\nabla u|^2}u\right) \\
     =&8\pi (u(\Sigma_+)-u(\Sigma_-)-\frac12u(\Sigma_+)\int_{\Sigma_+}H^2 +\frac12u(\Sigma_-)\int_{\Sigma_+}H^2
\end{align*}
which implies
$ m_H(\Sigma_+)\ge m_H(\Sigma_-) $,
where $m_H$ is the Hawking energy \cite{Hawking68}:
\begin{align*}
    m_H(\Sigma)=\sqrt{\frac{|\Sigma|}{16\pi}}\left(1-\frac1{16\pi}\int_\Sigma H^2\right).
\end{align*}
The monotonic property for the Hawking energy under inverse mean curvature flowed was first proved by Georch \cite{Geroch1973} in case there are no critical points and by Huisken-Ilmanen \cite{HuiskenIlmanen2001} in general under some topological assumptions.

\vspace{.2cm}

{\bf(II)}
{\sl $p$-harmonic functions and the monotonicity formulas}:
Let $1>U>0$ be a positive $p$-harmonic function.
In \cite{AMO2021, AMMO2022} Agostiniani-Mantegazza-Mazzieri-Oronzio showed the following monotonicity formula: For $0<t_1<t_2$
$$
F(t_2)\ge F(t_1)
$$
where $F(t)$ is given by \eqref{eq-F} for  $U$. In case of $|\nabla U|>0$, this can also be derived from Corollary \ref{c-integral formula-h}. In fact, let $u=(1-U)^{-\frac1a}$. Then $u$ satisfies \eqref{eq-bdry-system} for $\a=2-p$ by Lemma \ref{l-equations}. Moreover, $U=1-ct^{-a}$ if and only if $u=c^{-\frac1a}t$. Moreover at $\{U=f(t)\}$, $|\nabla u|=c^{-\frac1a}(ca)^{-1}t^{a+1}|\nabla U|$ and $\nabla u/|\nabla u|=\nabla U/|\nabla U|$. Hence apply Corollary \ref{c-integral formula-h} to $u$ with $\b=0$, we obtain the monotonicity of $F(t)$ for $U$.

\vspace{.2cm}

{\bf(III)}
{\sl Harmonic functions and the monotonicity of $D(t)$}:
Let $U$ be a positive harmonic function.
In \cite[Lemma 3.1]{Miao2022} the second author showed
 $$
 D(t_1)\ge D(t_2)
 $$
 for $0<t_1<t_2$. As in {\bf(II)} in case $|\nabla U|>0$, let $u=(1-U)^{-1}$, the above result is also a consequence of Corollary \ref{c-integral formula-h} with $\b=2$.

\begin{rem}\label{rem:all monotonicities}
Conceptually, it should not come as a surprise that Theorem \ref{T: integral formula-h-intro} implies various other monotonicity formulas which rely on Gauss-Bonnet's theorem.
Informally speaking, this requires the Gaussian curvature term appearing in a divergence identity such as equation \eqref{main formula} to be of the form $f(u)|\nabla u|K$ for some function $f$.
In order to ensure that the remaining terms on the left hand side of \eqref{main formula} are non-negative, the freedom of choosing $f$ is drastically restricted.
\end{rem}

\begin{rem}
For $\b=0$, Theorem \ref{T: integral formula-h-intro} also generalizes to initial data sets $(M,g,k)$ satisfying the dominant energy condition.
In this case the equation
\bee\label{eq:riem}
 \Delta u=\a\nabla^2u(\nu,\nu)+\frac{2|\nabla u|^2}u
\eee
needs to be replaced by the system
\begin{align*}
    \Delta u=& -\tr_g(k)|\nabla u| + \a\nabla^2u(\eta,\eta) +\a k(\eta,\eta)|\nabla u| +\frac{3|\nabla u||\nabla v|+\langle\nabla u,\nabla v\rangle}{u+v}  \\
\Delta v=&\tr_g(k)|\nabla v|+\a\nabla^2v(\eta,\eta)-\a k(\eta,\eta)|\nabla v| +\frac{3|\nabla u||\nabla v|+\langle\nabla u,\nabla v\rangle}{u+v}
\end{align*}
where $\eta=\frac{\nabla u|\nabla v|+\nabla v|\nabla u|}{|\nabla u|\nabla v|+\nabla v|\nabla u||}$.
Note that in case $k=0$, we can set $u=v$ and the above system reduces to equation \eqref{eq:riem}.
For more details see Theorem 1.1. and Corollary 1.2 in \cite{Hirsch2022}.
We believe that for $\b=\frac2{1-\a}$ there is no analogue for initial data sets satisfying the dominant energy condition.
\end{rem}

Let $(M^3,g)$ denote a complete manifold with compact boundary $\p M$. As mentioned in the introduction, $\cA, \cB, D, F, \cG$ are closely related.
Let $u$ be a $p$-harmonic function satisfying $ u \to 1 $ at $\infty$ and $ u < 1 $ on $M$.
To facilitate a comparison with the monotonicity for $p$-harmonic Green's function from \cite{CCLT2022, MunteanuWang2021},
we define, for the given $u$,
\be \label{eq-G-t}
\cG (t) = -4a^2\pi c t^{-a} + ( c t^{-a} )^{-1}   \int_{ \{u = f(t)\}  }|\nabla u |^2 .
\ee
Note that if $(M, g)$ is complete without boundary and $G$ is the positive $p$-harmonic Green's function with pole at $x_0$
and approaching $0$ at infinity, then choosing
$ u = 1 - G  \ \text{and} \ f (t) = 1 - t^{-a} , $
we have
\be
\begin{split}
\cG (t) = & \ -4a^2\pi  t^{-a} + (  t^{-a} )^{-1}   \int_{\{ 1 - G  = 1 -  t^{-a} \}  }|\nabla u |^2 \\
= & \ \cG (\tau), \ \ \ \ \ \text{upon a substitution} \ \tau = t^{-a} ,
\end{split}
\ee
the quantity given in \eqref{e-pG-intro}. We have:


\begin{lma}\label{l-combined-1} Suppose $|\nabla u |>0$. Then
\begin{enumerate}

  \item [(i)] $ D(t)= t^{-a-1} [ (1+2a) \cB(t)- a \cA(t )]  $. $F(t)=\cA(t)-\cB(t)$.
  \item[(ii)] $ D'(t)=  - a t^{-a -1} F'(t) $. Hence $D'\le 0$ if and only if $F'\ge0$.
  \item [(iii)] $\cG(t) = - c a^2 t^{-a-1} \cB(t) $.
  \item [(iv)] $\cG'(t) = c a^3 t^{-a - 2}  F (t)$. Hence $\cG'\ge0$ if and only if $F\ge0$.
      \item[(v)] $\cB'\ge 0$ if and only if $D\ge0$.
\end{enumerate}
\end{lma}
\begin{proof} (v) follows from the definition of $D$.

(i) $F(t)=\cA(t)-\cB(t)$ follows from their definitions. By \eqref{e-AB-1}, \eqref{eq-F}, and Lemma \ref{l-basic-2}
\bee
\begin{split}
\cB'(t)=&4\pi  -(ca)^{-2}(1+2a)t^{2a}\int_{ u = f(t) }|\nabla u|^2-(ca)^{-1} t^{a}\int_{ u = f(t) }(2\Delta u-|\nabla u|H)\\
=&4\pi  -(ca)^{-2}(1+2a)t^{2a}\int_{ u = f(t) }|\nabla u|^2+(ca)^{-1} at^{a}\int_{ u = f(t) }( |\nabla u|H \\
=&t^{-1}\lf((1+2a) \cB(t)- a \cA(t)\ri),
\end{split}
\eee
because $u$ is $p$-harmonic so that $\Delta u=\frac{2-p}{1-p}|\nabla u| H$. By the definition of $D$ in \eqref{e-D}, (i) follows. To prove (ii), by (i) we have
\bee
\begin{split}
D'(t)=&-(a+1)t^{-1}D(t) +t^{-a}\lf((1+2a) \cB'(t) - a \cA'(t)\ri)\\
=&-at^{-a}F'(t).
\end{split}
\eee

 (iii) follows from \eqref{eq-G-t}, \eqref{e-AB-1} and \eqref{eq-F}. To prove (iv):
 \bee
 \begin{split}
  \cG'(t) =&   c a^2(1+a) t^{-a-2} \cB(t)- c a^2 t^{-1} D(t)\\
  =& c a^2(1+a) t^{-a-2} \cB(t) -ca^2t^{-a-2} [ (1+2a) \cB(t) - a \cA(t )]\\
  =&c a^3 t^{-a - 2}  F (t).
  \end{split}
  \eee
\end{proof}

\begin{prop}\label{t-section-2}
Let $(M^3,g)$ be a complete Riemannian manifold with nonnegative scalar curvature, with compact boundary $\p M$.
For $ 1 < p \le 3$, suppose $u$ is a $p$-harmonic function on $M$ with $u=0$ at $\p M$ and $ u \to 1$ at infinity.
Suppose $ \p M$ is connected and $ | \nabla u | > 0 $.
Then the following monotonicity holds:
   \begin{enumerate}
     \item [(i)] (Local monotonicity) $D(t_1)\ge D(t_2)$ and $F(t_2)\ge F(t_1)$,  $ \forall \ t_1 < t_2 $.
     Moreover, if equality holds, then   $\{f(t_1)<u<f(t_2)\}$ is isometric to an annulus in $\R^3$.

     \vspace{.1cm}

     \item [(ii)] (Global monotonicity)
     If $(M, g)$ is asymptotically flat,
       then
       \begin{itemize}
       \item $D(t)\ge0$ and   $(1+2a) \cB(t) - a \cA(t)\ge 0 $;
       \item  $ \cA(t)$ and $ \cB(t)$ are monotone non-degreasing.
     If $\cB(t_1)=\cB(t_2)$ or $\cA(t_1)=\cA(t_2)$, then $\{u>f(t_1)\}$ is isometric to $\R^3$ outside a round ball.
     \end{itemize}

   \end{enumerate}

\end{prop}

\begin{proof}
(i) We have   seen $F'(t) \ge 0$ in Case ({\bf II}) after Corollary \ref{c-integral formula-h}.
It in turns shows (i) by Lemma \ref{l-combined-1} (ii). Or we can prove $D(t_1)\ge D(t_2)$ directly as follows. Let $w= \left(\frac{ 1 - u}{c} \right)^{-\frac1a}$, where $a=(3-p)/(p-1)$. Let $ \alpha = 2 - p $. Then $ w $ is a positive solution to
\bee
\Delta w = \alpha w_{\nu\nu}+2 w^{-1} |\nabla w|^2  .
\eee
Choose $\b=2/(1- \alpha )=2/(p-1)$, then $ 1 - \beta = -a $.
Given $t_1 < t_2$, by Corollary \ref{c-integral formula-h},
\bee
\begin{split}
0 \le & \ 4 \pi    \int_{t_1}^{t_2}\frac1{t^\b}
     + \lf(\int_{w = t_2 }-\int_{ w = t_1 }\ri)\frac1{t^\b}\left(-H_w|\nabla w|+\frac{2\b-1}{\b-1}\frac{|\nabla w|^2}w \right)  \\
= & - 4\pi a^{-1} (t_2^{-a} - t _1^{-a}) \\
& \ + \int_{\{ u = f(t_2) \}} \lf( - H_u  (ca)^{-1}  | \nabla u |+  (1+ 2 a)  a^{-1} (ca)^{-2} t_2^{ a } |\nabla u|^2 \ri) \\
& \ -  \int_{\{ u = f(t_1) \}} \lf( - H_u  (ca)^{-1}  | \nabla u |+  (1+ 2 a)  a^{-1} (ca)^{-2} t_1^{ a } |\nabla u|^2 \ri) \\
= & a^{-1} \left( D (t_1) - D(t_2) \right),
\end{split}
\eee
which proves (i). (The equality case follows from Corollary \ref{c-integral formula-h}.)
Note that, by Lemma \ref{l-combined-1} (ii), this implies $ F'(t) \ge 0 $.

(ii) Since $(M, g)$ is asymptotically flat, the asymptotical behavior of $u$ in \eqref{e-behavior} implies, as $ t \to \infty$,
$$ \int_{ u = f(t) } | \nabla u |^2 = O ( t^{-2a} ), \ \ \int_{ u = f(t) } H | \nabla u | = O ( t^{-a} ) .  $$
Therefore, $ D(t) = O ( t^{-a} ) $ by \eqref{e-D}. In particular, $ D(t) \to  0 $ as $ t \to \infty$. By (i), we have $ D (t) \ge 0 $.
As a result, by Lemma \ref{l-combined-1} (i), $( 1 + 2 a ) \cB(t) \ge \cA(t)$.
Also recall $ D(t) = t^{-a} \cB'(t)$ from \eqref{e-D}. Hence, $ D(t) \ge 0 $ shows $ \cB'(t) \ge 0 $, i.e. $\cB(t)$ is monotone non-decreasing.
Since  $ \cA (t) = \cB(t) + F(t) $, we see $\cA(t)$ is monotone non-decreasing as well.
\end{proof}

We may also obtain global monotone property for the Green's function using Lemma \ref{l-combined-1}.
 Suppose $G>0$ is the $p$-harmonic Green's function with pole at $x_0$ and $G\to 0$ at infinity.
Let $ u = 1 - G $ on $(M,g)$.
The behavior of $G(x)$ at the pole (see \cite{KV86, MRS19}) implies
$ F(t) \to 0 $ as $ t \to 0^+$.
By Case ({\bf II}) or Proposition \ref{t-section-2} (i), one knows $F'(t) \ge 0$.
Therefore, $ F(t) \ge 0 $.
By Lemma \ref{l-combined-1} (iv),  $ \cG'(t) \ge 0 $, which is equivalent to $\cG'(\tau) \le 0$.
This corresponds to the monotone property  of $G$ in  \cite{MunteanuWang2021, CCLT2022}. Hence we have:
\begin{prop}\label{p-Green} (Global monotonicity) Let $(M^3,g)$ be a complete Riemannian manifold. Suppose $M$ supports a positive $p$-harmonic Green's function $G$ with pole at $x_0$ and $G\to 0$ at infinity. Assume that $|\nabla G|>0$ on $M\setminus\{x_0\}$. Then $\cG'(\tau)\le0$, where $\cG(\tau)$ is as in \eqref{e-pG-intro}.
\end{prop}

\begin{rem}
    It is worthy of writing $ \cB(t)$, $ \cA(t)$ directly via the level set $\{ u = s \}$ and the parameter $s$.
Let
$ t = \left( \frac{c}{1- s} \right)^\frac{1}{a} , $
$\cB(t)$ and $\cA(t)$ take the form of
\be
\cB( s ) = \left( \frac{c}{1- s} \right)^\frac{1}{a}  \left[ 4 \pi - \frac{1}{ a^2 ( 1 - s )^2 } \int_{ u = s } | \nabla u |^2 \right]
\ee
and
\be
\cA ( s ) = \left( \frac{c}{1- s} \right)^\frac{1}{a}  \left[ 8 \pi - \frac{1}{ a ( 1 - s ) } \int_{ u = s } | \nabla u | H \right] .
\ee
If $p=2$, then $a=1$,  these were given in \cite{Miao2022}.
A feature of such expressions is that they pinpoint a fact
$$ u^{(s)} : = \frac{ u - s}{ 1 - s } $$
is the p-harmonic function, vanishing on $\{ u = s \}$ and approaching $1$ at $\infty$.
This indicates $\cB (\cdot ) $ and $\cA (\cdot) $ are suitable scalings of
$$ 4 \pi - a^{-2} \int_{\p M_s } | \nabla u^{(s)} |^2  \ \  \ \text{and} \ \ \ 8 \pi - a^{-1}  \int_{\p M_s } | \nabla u^{(s)}  | H  $$
on $M_s = \{ u \ge s \}$.
\end{rem}

\section{Monotonicity of $\cA(t), \cB(t), D(t)$ via regularization}\label{s-regular}

We will prove Theorem  \ref{t-intro-1} stated in the introduction.
First, we point out that it suffices to establish the monotonicity of the quantities involved.
Once the monotonicity is shown, the equality case will follow by applying Corollary \ref{c-integral formula-h},
starting from the boundary $ \p M = \{ u = 0 \}$ which is a regular level set of $u$.

We recall the setting of Theorem \ref{t-intro-1}:
$(M^3,g)$ is a complete three manifold with smooth compact boundary $\p M$ such that
\begin{enumerate}
  \item [(i)] $\p M$ is connected;
  \item [(ii)] $H_2(M,\p M)=\{0\}$;
  \item [(iii)] $M$ has one end which is asymptotically flat; and
\item [(iv)] the scalar curvature $\cS$ of $M$ is nonnegative.
\end{enumerate}
Let $u$ be the solution of the $p$-harmonic equation so that $u=0$ on $\p M$ and $u\to 1$ near infinity. As before, we denote $ f(t)=f_0(t)$ in the theorem by
$$
f_0(t)=1-ct^{-a},
$$
where $a=(3-p)/(p-1)$,  $c =a^{-1}\lf(C_p/(4\pi)\ri)^\frac1{p-1}$, and $C_p$ is the $p$-capacity of $\p M$ in $M$. Given any $T>0 $ with $0<f_0(T)<1$,  the level set  $ \{u=f_0(T)\}$ is compact and will not intersect $\p M$ by strong maximum principle \cite[Theorem 6.5]{HKM93}. Assume this  level set  is a regular level set of $u$. We approximate $u$ by smooth functions. Following \cite{AMMO2022}, for any $\e>0$, let $v=v_\e$ be the solution of
\be\label{e-v-1}
\left\{
  \begin{array}{ll}
    \div(|\nabla v|_\e^{p-2}\nabla v)=0, & \hbox{in $M(T)$};\\
v=0&\hbox{on  $\p M$};\\
v=f_0(T)&\hbox{on  $\Sigma(T)$};
  \end{array}
\right.
\ee
where $M(T)=\{0<u<f_0(T)\}$, $\Sigma(t)=\{u=f_0(t)\}$, and for any $\eta>0$, 

\be
|\nabla v|_\eta=\sqrt{|\nabla v|^2+\eta^2}.
 \ee
 Then $v_\e$ is smooth. As $\e\to0$, $v_\e\to u$ in $C^{1,\beta}$ norm for some $\beta>0$, and $v_\e\to u$ in $C^\infty$ norm near the points where $|\nabla u|>0$ by \cite{Benedetto1,Benedetto2}.
Define
\be\label{e-capacity-1}
\left\{
  \begin{array}{ll}
   C_{p,\e}=\int_{\p M}|\nabla v|_\e^{p-2}|\nabla v|;\\
c_\e=a^{-1}\lf(\frac{C_{p,\e}}{4\pi}\ri)^{\frac1{p-1}}.\\
  \end{array}
\right.
\ee
Note that as $\e\to0$, $c_\e\to c$, $C_{p,\e}\to C_p$.
For $0<t<T$, let $f_\e(t)=1-c_\e t^{-a}$ and let  $\Sigma(\e,t)=\{v=v_\e=f_\e(t)\}$, which will be in the interior of $M(T)$, provided $\e$ is small enough.
Observe that $C_{p,\e}=\int_{\Sigma(\e,t)}|\nabla v|_\e^{p-2}|\nabla v|$  whenever $\Sigma(\e,t)$ is regular.
If $\Sigma(\e,t)$ is regular, define
  corresponding $D_\e(t)$ and $ \cB_\e(t)$ as follows:
\be\label{e-De}
\begin{split}
D_\e(t)=&4\pi t^{-a}  -(c_\e a)^{-2}(1+2a)t^{a}\int_{\Sigma(\e,t)}|\nabla v|^2+c^{-1}_\e   \int_{\Sigma(\e,t)}|\nabla v|H,
\end{split}
\ee
\be\label{e-Be}
\cB_\e(t)=4\pi t-(c_\e a)^{-2}t^{2a+1}\int_{\Sigma(\e,t)}|\nabla v_\e|^2.
\ee
Here $H=H_v$ is the mean curvature with respect to $\nu=\nabla v/|\nabla v|$. The corresponding $f, \cB, D$ for $u$ will also be denoted by $f_0, \cB_0, D_0$ etc.
By the proof of  \cite[Lemma 1.3]{AMMO2022}, the following fact is true:
\begin{lma}\label{l-approx} Suppose $\{u=f_0(t)\}$ is regular with $0<t<T$ which implies that  $0<f_0(t)< f_0(T)$. Then for $\e>0$ small enough, $\Sigma(\e,t)=\{v_\e=f_\e(t)\}$ is also regular. Moreover,
\bee
\lim_{\e\to 0}D_\e(t)=D_0(t); \ \lim_{\e\to 0} \cB_\e(t) = \cB_0(t).
\eee

\end{lma}

To simplify notation, in what follows,   whenever there is no confusion, we will suppress the index $\e$.
For example, we denote $c_\e$ by $c$ and the original $c$, which is the limit of $c_\e$ as $\e\to0$, will be denoted by $c_0$ instead. Direct computations give:

\begin{lma}\label{l-acomputation}
Suppose $\Sigma(\e,t)$ is regular, then at this level set:
\bee
\left\{
  \begin{array}{ll}
    \Delta v=(2-p)\frac{|\nabla v|^2}{|\nabla v|_\e^2}v_{\nu\nu}\\
H=\frac1{|\nabla v|}(\Delta v-v_{\nu\nu})=-\frac1{|\nabla v|}\frac{(p-1)|\nabla v|^2+\e^2}{|\nabla v|_\e^2} v_{\nu\nu}
  \end{array}
\right.
\eee
where $$v_{\nu\nu}=\nabla^2v(\nu,\nu)=\frac{\la\nabla |\nabla v|,\nabla v\ra}{|\nabla v|}= \la\nabla |\nabla v|,\nu\ra$$
and $\nu=\nabla v/|\nabla v|$.
\end{lma}
Hence if $\Sigma(\e,t)$ is regular, then
\be\label{e-De-1}
\begin{split}
D_\e(t)
=&c^{-1}  \int_{\Sigma(\e,t)}\bigg(4\pi C_{p,\e}^{-1}|\nabla v|^{p-2}_\e |\nabla v|   (1-v)  - a^{-2}(1+2a)(1-v)^{-1} |\nabla v|^2\\
&\ \ \ \ \ \ \ \ \ \ \ \   +\Delta v-v_{\nu\nu}\bigg).
\end{split}
\ee
  Let
\be\label{e-Xe-1}
X=X_\e=W-U+V
\ee
where
\bee
\left\{
  \begin{array}{ll}
    W=4\pi C_{p,\e}^{-1} (1-v)|\nabla v|^{p-2}_\e  \nabla v;\\
U=a^{-2}(1+2a)(1-v)^{-1} |\nabla v|\nabla v;\\
V=\frac{\Delta v}{|\nabla v|}\nabla v-  \nabla |\nabla v|.
  \end{array}
\right.
\eee
Since $U$ or $V$ may not be defined or smooth if $|\nabla v|=0$, we further regularize these functions as follows.
For, $\delta>0$, let
 \bee
\left\{
  \begin{array}{ll}
    U_\delta=a^{-2}(1+2a)(1-v)^{-1} |\nabla v|_\delta\nabla v,\\
V_\delta=\frac{\Delta v}{|\nabla v|_\delta}\nabla v-  \nabla |\nabla v|_\delta.
  \end{array}
\right.
\eee
 Then $W, U_\delta, V_\delta$ are smooth.
Suppose $0<t_1<t_2<T$ so that $  \Sigma(\e,t_1), \Sigma(,\e, t_2)$ are regular. One can see that
\be\label{e-limit-D}
D_\e(t_2)-D_\e(t_1)=\lim_{\delta\to0}\int_{\{f_\e(t_1)<v<f_\e(t_2)\}}\div (W-U_\delta+V_\delta).
\ee
\begin{lma}\label{l-div-WUV} \begin{enumerate}\
                               \item [(i)]

 \bee
\div \, W=-4\pi C_{p,\e}^{-1}|\nabla v|^{p-2}_\e|\nabla v|^2.
\eee
                               \item [(ii)] At $|\nabla v|=0$,
\bee
\div\, U_\delta=\delta a^{-2}(1+2a)(1-v)^{-1} \Delta v.
\eee
At $|\nabla v|>0$,
\bee
\begin{split}
\div\, U_\delta =&a^{-2}(1+2a)\lf( (2-p)\frac{|\nabla v|_\delta |\nabla v|^2}{|\nabla v|_\e^2}+\frac{|\nabla v|^2}{|\nabla v|_\delta}\ri)(1-v)^{-1}v_{\nu\nu}\\
&+a^{-2}(1+2a)(1-v)^{-2}|\nabla v|^2|\nabla v|_\delta.
\end{split}
\eee

 \item [(iii)] At $|\nabla v|=0$, $\div\, V_\delta \le0$. At $|\nabla v|>0$,
     \bee\begin{split}
     \div\, V_\delta\le&
     \lf(\frac{(2-p)^2|\nabla v|^4}{|\nabla v|_\e^4|\nabla v|_\delta}
-\frac{(2-p)|\nabla v|^4}{|\nabla v|_\e^2|\nabla v|_\delta^3}\ri)v_{\nu\nu}^2
 -|\nabla v|_\delta^{-1}|\nabla v|^2\lf(|A|^2+\Ric(\nu,\nu)\ri), \\
\end{split}
\eee
where $A$ the second fundamental form of the level surface with respect to the unit normal $\nu$.

                             \end{enumerate}

\end{lma}
\begin{proof} (i) Since $v$ satisfies \eqref{e-v-1}
\be\label{e-W}
\begin{split}
\div \, W=
&\div\lf((1-v)|\nabla v|^{p-2}_\e\nabla v\ri)\\
=&4\pi C_{p,\e}^{-1}\lf[(1-v)\div\lf( |\nabla v|^{p-2}_\e\nabla v\ri)-\la \nabla v, |\nabla v|^{p-2}_\e\nabla v\ra\ri]\\
=&-4\pi C_{p,\e}^{-1}|\nabla v|^{p-2}_\e|\nabla v|^2.
\end{split}
\ee

(ii) At $|\nabla v|=0$, it is easy to see that
\bee
\div\, U_\delta=\delta a^{-2}(1+2a)(1-v)^{-1} \Delta v.
\eee
At $|\nabla v|>0$, by Lemma \ref{l-acomputation}, we have
\bee
\begin{split}
\div\, U_\delta=&a^{-2}(1+2a)\bigg( (1-v)^{-1}|\nabla v|_\delta \Delta v+(1-v)^{-2}|\nabla v|^2|\nabla v|_\delta\\
&+(1-v)^{-1}\frac{|\nabla v|}{|\nabla v|_\delta}\la \nabla |\nabla v|,\nabla v\ra\bigg)\\
=&a^{-2}(1+2a)\bigg( (1-v)^{-1}|\nabla v|_\delta \Delta v+(1-v)^{-2}|\nabla v|^2|\nabla v|_\delta\\
&+(1-v)^{-1}\frac{|\nabla v|^2}{|\nabla v|_\delta}\nabla^2v(\nu,\nu) \bigg)\\
=&a^{-2}(1+2a)\bigg( (2-p)(1-v)^{-1}\frac{|\nabla v|_\delta |\nabla v|^2}{|\nabla v|_\e^2}v_{\nu\nu}\\
&+(1-v)^{-2}|\nabla v|^2|\nabla v|_\delta+(1-v)^{-1}\frac{|\nabla v|^2}{|\nabla v|_\delta}v_{\nu\nu} \bigg).
\end{split}
\eee

(iii)  To compute $\div\, V_\delta$,  by Lemma \ref{l-acomputation} we have
\bee
\begin{split}
\div\lf(\frac{\Delta v}{|\nabla v|_\delta}\nabla v\ri)
=&\frac{(\Delta v)^2}{|\nabla v|_\delta}+\frac{\la \nabla \Delta v, \nabla v\ra}{|\nabla v|_\delta}- \frac12\frac{\Delta v}{|\nabla v|_\delta^3}\cdot  \la\nabla(|\nabla v|^2), \nabla v\ra\\
=&(2-p)^2\frac{|\nabla v|^4}{|\nabla v|_\e^4|\nabla v|_\delta}v_{\nu\nu}^2
-(2-p)\frac{|\nabla v|^4}{|\nabla v|_\e^2|\nabla v|_\delta^3}v_{\nu\nu}^2\\
&+\frac{\la \nabla \Delta v, \nabla v\ra}{|\nabla v|_\delta},\\
\end{split}
\eee
which is zero at the points where $|\nabla v|=0$ because $\Delta v$ will be zero at this point by \eqref{e-v-1}.

\bee
\begin{split}
\div(\nabla |\nabla v|_\delta)=&\Delta |\nabla v|_\delta\\
 =&\frac12|\nabla v|_\delta^{-1}\lf(\Delta |\nabla v|_\delta^2-2|\nabla |\nabla v|_\delta|^2\ri)\\
=& |\nabla v|_\delta^{-1}\lf(|\nabla^2v|^2+\Ric(\nabla v,\nabla v)+\la \nabla v,\nabla \Delta v\ra-|\nabla |\nabla v|_\delta|^2\ri)\\
=&|\nabla v|_\delta^{-1}\lf(|\nabla^2v|^2+\Ric(\nabla v,\nabla v)\ri)\\&+ \frac{\la \nabla \Delta v,\nabla v\ra}{|\nabla v|_\delta}-\frac14\frac{1}{|\nabla v|_\delta^3}|\nabla (|\nabla v|^2)|^2.
\end{split}
\eee
which is nonnegative at the points $|\nabla v|=0$ because $\Delta |\nabla v|_\delta\ge0$.

Hence if $|\nabla v|=0$, then $\div\,V_\delta\le0$. If $|\nabla v|>0$, then
\bee
\begin{split}
\div\, V_\delta=&(2-p)^2\frac{|\nabla v|^4}{|\nabla v|_\e^4|\nabla v|_\delta}v_{\nu\nu}^2
-(2-p)\frac{|\nabla v|^4}{|\nabla v|_\e^2|\nabla v|_\delta^3}v_{\nu\nu}^2+\frac{|\nabla v|^2}{|\nabla v|_\delta^3} |\nabla|\nabla v||^2\\
&-|\nabla v|_\delta^{-1}\lf(|\nabla^2v|^2+\Ric(\nabla v,\nabla v)\ri).
\end{split}
\eee
For $|\nabla v|>0$, $|\nabla |\nabla v||^2=\lf(|\wn |\nabla v||^2+v_{\nu\nu}^2\ri)$, where $\wn$ is the derivative on the level set. Hence if $|\nabla v|>0$, then
\be\label{e-V-delta}
\begin{split}
\div\, V_\delta=&(2-p)^2\frac{|\nabla v|^4}{|\nabla v|_\e^4|\nabla v|_\delta}v_{\nu\nu}^2
-(2-p)\frac{|\nabla v|^4}{|\nabla v|_\e^2|\nabla v|_\delta^3}v_{\nu\nu}^2+\frac{|\nabla v|^2}{|\nabla v|_\delta^3 }(|\wn |\nabla v||^2+v_{\nu\nu}^2) \\
&-|\nabla v|_\delta^{-1}\lf(|\nabla v|^2|A|^2+2|\wn|\nabla v||^2+v_{\nu\nu}^2+|\nabla v|^2\Ric(\nu,\nu)\ri)\\
\le& \lf((2-p)^2\frac{|\nabla v|^4}{|\nabla v|_\e^4|\nabla v|_\delta}
-(2-p)\frac{|\nabla v|^4}{|\nabla v|_\e^2|\nabla v|_\delta^3}\ri)v_{\nu\nu}^2  \\
&- |\nabla v|_\delta^{-1}|\nabla v|^2\lf(|A|^2+\Ric(\nu,\nu)\ri).
\end{split}
\ee
This completes the proof of the lemma.
\end{proof}

\begin{lma}\label{l-monotone-d} With the   assumptions and notation as in Lemma \ref{l-div-WUV}, assume $1<p\le 2$, and $0<t_1<t_2<T$ so that $\{u=f_0(t_1)\}, \{u=f_0(t_2)\}$ are regular. We have the following:
\begin{enumerate}
  \item [(i)] For $\e>0$ small enough, so that $\Sigma(\e,t_1)$, $\Sigma(\e,t_2)$ are regular, we have
  \bee
 D_\e(t_2)-D_\e(t_1) \le  \int_{\{f_\e(t_1)<v<f_\e(t_2)\}}E(\e)
\eee
where $E(\e)\ge0$ is a continuous function which is uniformly bounded independent of $\e$   and $E(\e)\to 0$ everywhere as $\e\to0$ in $M(T)$.

  \item [(ii)]   There is a constant $C$ independent of $\e$ such that
  \bee
  \int_{\{f_\e(t_1)<v<f_\e(t_2), |\nabla v|>0\}} \frac{|\Delta v|}{|\nabla v|} \le C\e^{-1}.
  \eee
\end{enumerate}

\end{lma}
\begin{proof} (i) We remark  that since $v=v_\e$ converges in $C^{1,\beta}(M(T))$ to $u$, $\Sigma(\e,t_1)$, $\Sigma(\e,t_2)$ are regular surfaces in the interior of $M(T)$,  provided $\e>0$ is small enough.

In the following, we denote $f_\e$ by $f$ for simplicity. By \eqref{e-limit-D},
\be \label{e-limit-d-1}
D_\e(t_2)-D_\e(t_1)=\lim_{\delta\to0}\int_{\{f(t_1)<v<f(t_2)\}}\div (W-U_\delta+V_\delta).
\ee
To prove (i), we need to estimate the above limit.  By Lemma \ref{l-div-WUV}, for any $\delta>0$,
\bee
\begin{split}
&\int_{\{f(t_1)<v<f(t_2)\}}\div(W-U_\delta+V_\delta)\\
\le& \int_{\{f(t_1)<v<f(t_2)\}}\div\,W
+C\delta \\
&-\int_{\{f(t_1)<v<f(t_2)\}}\mathbf{1}_{\{|\nabla v|>0\}}\div\, U_\delta
 +\int_{\{f(t_1)<v<f(t_2)\}}\mathbf{1}_{\{|\nabla v|>0\}}(\mathrm{I}_\delta-\mathrm{II}_\delta) .
\end{split}
\eee
for some $C>0$ independent of $\delta$, where $\mathbf{1}_Y$ is the characteristic function of the set $Y$, and
\bee
\left\{
  \begin{array}{ll}
    \mathrm{I}_\delta= \lf((2-p)^2\frac{|\nabla v|^4}{|\nabla v|_\e^4|\nabla v|_\delta}
-(2-p)\frac{|\nabla v|^4}{|\nabla v|_\e^2|\nabla v|_\delta^3}\ri)v_{\nu\nu}^2
-|\nabla v|_\delta^{-1}|\nabla v|^2\Ric(\nu,\nu) \\
   \mathrm{II}_\delta=|\nabla v|_\delta^{-1}|\nabla v|^2|A|^2.
  \end{array}
\right.
\eee
 Since $v$ is smooth, $v_{\nu\nu}$ is uniformly bounded on  $ \mathbf{1}_{\{|\nabla v|>0\}}$ and so $\mathbf{1}_{\{|\nabla v|>0\}}\div U_\delta $ is uniformly bounded. Similarly, $\mathbf{1}_{\{|\nabla v|>0\}}\mathrm{I}_\delta$   is uniformly bounded because Ricci  curvature is bounded. Also $ \mathrm{II}_\delta$ is nonnegative and is nondecreasing as $\delta \to0$, by dominated and monotone convergence theorems, we have:
\be\label{e-d-to-0}
\begin{split}
&\lim_{\delta\to0}\int_{f(t_1)<v<f(t_2)}\div (W-U_\delta+V_\delta)\\
\le&\int_{\{f(t_1)<v<f(t_2)\}}\div\,W+\int_{\{f(t_1)<v<f(t_2)\}}\mathbf{1}_{\{|\nabla v|>0\}}(-\div \, U +\mathrm{I}_0-\mathrm{II}_0) \\
=:&\mathrm{III}.
\end{split}
\ee
 Here in $\{|\nabla v|>0\}$,
$$
\mathrm{I}_0=\lim_{\delta\to0}\mathrm{I}_\delta;\ \ \mathrm{II}_0=\lim_{\delta\to0}\mathrm{II}_\delta.
$$
In particular, since the LHS in \eqref{e-limit-d-1} is finite, the integrals of $\div\, W $ and $\mathbf{1}_{\{|\nabla v|>0\}}(-\div \, U +\mathrm{I}_0)$ are finite, we have
\be\label{e-A-int}
\int_{\{f(t_1)<v<f(t_2)\}}\mathbf{1}_{\{|\nabla v|>0\}}|\nabla v||A|^2<\infty.
\ee
Let
$$\tau_1=f(t_1), \tau_2=f(t_2), \Sigma_\tau=\{v=\tau\}.
$$
By the co-area formula   and   the Morse-Sard's theorem, which implies the set of critical values of $v$ is of measure zero in $[\tau_1, \tau_2]$,  for any $L^1$ function $h$ in $\{\tau_1<v<\tau_2\}$
$$
\int_{\{\tau_1<v<\tau_2\}}|\nabla v|h=\int_{\{\tau_1<\tau<\tau_2, \tau\in R\}} \lf(\int_{\Sigma_\tau}h\ri)d\tau
$$
where $R$ is the set of regular values of $v$. We now apply this to different terms in $\mathrm{III}$.
\bee
\begin{split}
\int_{\{f(t_1)<v<f(t_2)\}}\div\,W=&
\int_{\{f(t_1)<v<f(t_2)\}}\lf(-4\pi C_{p,\e}^{-1} |\nabla v|^2|\nabla v|_\e^{p-2}\ri)\\
=&\int_{\{\tau_1<\tau<\tau_2, \tau\in R\}}\int_{\Sigma_\tau} \lf(-4\pi C_{p,\e}^{-1} |\nabla v||\nabla v|_\e^{p-2}\ri)d\tau\\
=&-\int_{\{\tau_1<\tau<\tau_2, \tau\in R\}}4\pi d\tau
\end{split}
\eee
On the other hand, in the set $ \{|\nabla v|>0\}$,
\be \label{e-U-delta-1}
|\nabla v|^{-1}\div\,U=a^{-2}(1+2a) \bigg[\lf((2-p)\frac{|\nabla v|^2}{|\nabla v|_\e^2}+1\ri)(1-v)^{-1}v_{\nu\nu}+(1-v)^{-2}|\nabla v|^2\bigg].
\ee
So $\mathbf{1}_{\{|\nabla v|>0\}}|\nabla v|^{-1}\div\,U $ is uniformly bounded because $v$ is smooth and $v\le f(t_2)\le f(T)<1$.
Hence
\bee
\int_{\{f(t_1)<v<f(t_2)\}}\mathbf{1}_{\{|\nabla v|>0\}} \div \, U=
\int_{\{\tau_1<\tau<\tau_2, \tau\in R\}}\int_{\Sigma_\tau}\mathbf{1}_{\{|\nabla v|>0\}}|\nabla v|^{-1}\div\,U.
\eee
Similarly, we also have
\bee
\int_{\{f(t_1)<v<f(t_2)\}}\mathbf{1}_{\{|\nabla v|>0\}} \mathrm{I}_0
=\int_{\{\tau_1<\tau<\tau_2, \tau\in R\}}\int_{\Sigma_\tau}\mathbf{1}_{\{|\nabla v|>0\}}|\nabla v|^{-1}\mathrm{I}_0.
\eee
Let  $\phi=:|A|^2 \mathbf{1}_{\{|\nabla v|>0\}}$. Since $|\nabla v|$ is continuous,  $|A|^2$ is smooth in  $\{|\nabla v|>0\}$, $\phi$ is measurable.  Since $\phi\ge0$, one can apply co-area formula to $\phi_k=\min\{\phi,k\}$ for $k\in \mathbb{N}$ so that
 $$
 \int_{\{f(t_1)<v<f(t_2)\}}|\nabla v|\phi_k=\int_{\{\tau_1<\tau<\tau_2, \tau\in R\}}\int_{\Sigma_\tau} \phi_k d\tau.
 $$
 Since $\phi_k\uparrow\phi$, one can apply monotone convergence theorem to both sides to conclude that
 $$
  \int_{\{f(t_1)<v<f(t_2)\}}|\nabla v|\phi =\int_{\{\tau_1<\tau<\tau_2, \tau\in R\}}\int_{\Sigma_\tau} \phi d\tau.
 $$
 However, by \eqref{e-A-int}, $|\nabla v|\phi=|\nabla v||A|^2\mathbf{1}_{\{|\nabla v|>0\}}=\mathrm{II}_0$ is integrable, we have
 \bee
\int_{\{f(t_1)<v<f(t_2)\}}\mathbf{1}_{\{|\nabla v|>0\}} \mathrm{II}_0
=\int_{\{\tau_1<\tau<\tau_2, \tau\in R\}}\int_{\Sigma_\tau}\mathbf{1}_{\{|\nabla v|>0\}}|\nabla v|^{-1}\mathrm{II}_0
\eee
which is finite.

Hence

\be\label{e-III}
\begin{split}
\mathrm{III}=&\int_{\{\tau_1<\tau<\tau_2, \tau\in R\}}\lf[ -4\pi+\int_{\Sigma_\tau}
( -|\nabla v|^{-1}\div\,U+|\nabla v|^{-1}\lf(\mathrm{I}_0-\mathrm{II}_0\ri))\ri]d\tau.
\end{split}
\ee
Here we have used the fact that $\mathbf{1}_{\{|\nabla v|>0\}}=1$ on $\Sigma_\tau$ for $\tau\in R$.
      Recall that,   for $1<p\le 2$,   if $|\nabla v|>0$, then
$$
v_{\nu\nu}=-\frac{|\nabla v|\,|\nabla v|^2_\e}{(p-1)|\nabla v|^2+\e^2}H.
$$
By \eqref{e-U-delta-1}, let $\lambda>0$ be a positive function to be determined later, if $|\nabla v|>0$, we have
\be\label{e-U-delta}
\begin{split}
&|\nabla v|^{-1}\div\,U\\
=&a^{-2}(1+2a) \bigg[-\lf((2-p)\frac{|\nabla v|^2}{|\nabla v|_\e^2}+1\ri)(1-v)^{-1}\frac{|\nabla v|\,|\nabla v|^2_\e}{(p-1)|\nabla v|^2+\e^2}H\\
&+(1-v)^{-2}|\nabla v|^2\bigg]\\
=&a^{-2}(1+2a) \bigg[-\lf( \frac{(3-p)|\nabla v|^2+\e^2}{(p-1)|\nabla v|^2+\e^2} \ri)(1-v)^{-1}|\nabla v|H +(1-v)^{-2}|\nabla v|^2\bigg]\\
\ge &-\frac12(1+2a)\lambda H^2 \\&+a^{-2}(1+2a)\lf(1-\frac12\lambda^{-1}\lf( \frac{|\nabla v|^2+(3-p)^{-1}\e^2}{|\nabla v|^2+(p-1)^{-1}\e^2} \ri)^2\ri)(1-v)^{-2}|\nabla v|^2\\
=&-\frac12(1+2a)\lambda H^2 +a^{-2}(1+2a)\lf(1-Q\ri)(1-v)^{-2}|\nabla v|^2
\end{split}
\ee
where
\be\label{e-Q}
Q=:\frac12\lambda^{-1}\lf( \frac{|\nabla v|^2+(3-p)^{-1}\e^2}{|\nabla v|^2+(p-1)^{-1}\e^2} \ri)^2.
\ee

At  the points where $|\nabla v|>0$, for $1<p\le 2$, we have

\be\label{e-I-II-2}
\begin{split}
|\nabla v|^{-1}(\mathrm{I}_0-\mathrm{II}_0)=  &|\nabla v|^{-2}\bigg[(2-p)^2\frac{|\nabla v|^4}{|\nabla v|_\e^4}-(2-p)\frac{|\nabla v|^2}{|\nabla v|_\e^2 }\bigg]v_{\nu\nu}^2-  (|A|^2+\Ric(\nu,\nu) )\\
\le & (2-p)(1-p)\frac{|\nabla v|^2}{|\nabla v|_\e^4}   v_{\nu\nu}^2-\frac34H^2  + K\\
= &  -\lf[\frac{(2-p)(p-1)|\nabla v|^4}{((p-1)|\nabla v|^2+ \e^2)^2} +\frac34\ri]H^2+K.\\
=:&-PH^2+K,
\end{split}
\ee
where $K$ is the Gaussian curvature of the level set and $P>0$ is the function inside the square bracket. Here we have used the facts that   $\cS\ge0$, 
$$
\Ric(\nu,\nu)=\frac12(\cS-2K+H^2-|A|^2)\ge -K+\frac12(H^2-|A|^2)
$$
and that 
$$
|A|^2=|\mathring A|^2+\frac12 H^2\ge\frac12 H^2,
$$
where $\mathring A$ is the traceless part of $A$.

Since $H_2(M,\p M)=\{0\}$ and since $\p M$ is connected, the level-set $\Sigma_\tau$ is also connected according to \cite[pages 9-10]{AMMO2022}.
Consequently, we have $\int_{\Sigma_\tau}K\le 4\pi$ for all regular $\tau$.
Hence for $\tau\in R$, if we choose $\lambda=2(1+2a)^{-1}P$, we have
\bee
\begin{split}
-4\pi+&\int_{\Sigma_\tau}
( -|\nabla v|^{-1}\div\,U+|\nabla v|^{-1}\lf(\mathrm{I}_0-\mathrm{II}_0\ri)\\
\le& \int_{\Sigma_\tau}(\frac12\lambda (1+2a)-P)H^2-a^{-2}(1+2a)(1-v)^{-2}|\nabla v|^2 (1-Q)\\
=& -\int_{\Sigma_\tau} a^{-2}(1+2a)(1-v)^{-2}|\nabla v|^2 (1-Q).
\end{split}
\eee

Combining with \eqref{e-III},  using co-area formula and Moser-Sard's theorem again ($Q$ is uniformly bounded, see below):
\be\label{e-III-2}
\begin{split}
\mathrm{III}\le& -\int_{\{\tau_1<v<\tau_2, \tau\in R\}}\int_{\Sigma_\tau} a^{-2}(1+2a)(1-v)^{-2}|\nabla v|^2 (1-Q)\\
 =&-\int_{\{f(t_1)<v<f(t_2)\}} a^{-2}(1+2a)(1-v)^{-2}|\nabla v|^3 (1-Q),
\end{split}
\ee
Since   $P\ge \frac34$ and $1<p\le 2$ so that
$$
(1+\frac{3-p}{p-1})^{-1}\le \frac{|\nabla v|^2+(3-p)^{-1}\e^2}{|\nabla v|^2+(p-1)^{-1}\e^2}\le 1,
$$
 from  the definition of $Q$ in \eqref{e-Q}, we conclude that $Q$ is uniformly bounded independent of $\e$.
Since $v\le f_0(T)<1$, $v\to u$ in $C^{1,\beta}$ norm for some $\beta>0$, we have $|\nabla v|\to |\nabla u|$ as $\e\to0$.  At the points where $|\nabla u|=0$, then one conclude that
$$
(1-v)^{-2}|\nabla v|^3 (1-Q)\to 0
$$
as $\e\to 0$. At the points where $|\nabla u|>0$, as $\e\to 0$ we have
$$
Q=\frac14(1+2a)\lf[\frac{(2-p)(p-1)|\nabla v|^4}{((p-1)|\nabla v|^2+ \e^2)^2} +\frac34\ri]^{-1}\lf(\frac{|\nabla v|^2+(3-p)^{-1}\e^2}{|\nabla v|^2+(p-1)^{-1}\e^2}\ri)^2 \to 1.
$$
Hence we also have:
$$
(1-v)^{-2}|\nabla v|^3 (1-Q)\to 0.
$$
From this, the fact that $v<f(T)<1$, $v$ converges in $C^{1,\beta}$ to $u$, \eqref{e-limit-D}, \eqref{e-d-to-0} and \eqref{e-III-2}  we conclude that (i) is true with
$$
E(\e)= a^{-2}(1+2a)(1-v)^{-2}|\nabla v|^3 |1-Q|.
$$

To prove (ii), if $p=2$, then $\Delta v=0$ and the result is obvious.
Observe that  as $\e\to0$,

\bee
\begin{split}
\int_{\{f(t_1)<v<f(t_2)\}}\div\, V_\delta=&\int_{\Sigma(t_2)}\la V_\delta, \frac{\nabla v}{|\nabla v|}\ra-\int_{\Sigma(t_1)}\la V_\delta, \frac{\nabla v}{|\nabla v|}\ra\\
\to&\int_{\Sigma(t_2)}(\Delta v-v_{\nu\nu}) -\int_{\Sigma(t_1)}(\Delta v-v_{\nu\nu})\\
\ge&-C_1
\end{split}
\eee
for some constant $C_1>0$ independent of $\e$, because  $\{u=f_0(t_1)\}, \{u=f_0(t_2)\}$ are regular and $v=v_\e$ converge in $C^\infty$ norm near these two level sets \cite{AMMO2022}.
 For $1<p<2$, by Lemma \ref{l-div-WUV}(iii),   for $|\nabla v|=0$, $\div\, V_\delta\le0$, so
\bee\begin{split}
                                  \int_{\{f(t_1)<v<f(t_2)\}}  \div\, V_\delta
                                \le&
                                  \int_{\{f(t_1)<v<f(t_2), |\nabla v|>0\}} \bigg[\lf(\frac{(2-p)^2|\nabla v|^4}{|\nabla v|_\e^4|\nabla v|_\delta}
-\frac{(2-p)|\nabla v|^4}{|\nabla v|_\e^2|\nabla v|_\delta^3}\ri)v_{\nu\nu}^2
\\&-|\nabla v|_\delta^{-1}|\nabla v|^2\Ric(\nu,\nu)\bigg] .
\end{split}
\eee

 On the other hand,
$$
\lf||\nabla v|_\delta^{-1}|\nabla v|^2\Ric(\nu,\nu)\ri|\le C_2
$$
by a constant $C_2$ independent of $0<\delta, \e\le 1 $. Hence, letting $\delta \to0$,
by the monotone convergence theorem, we have
\bee
\begin{split}
-C_3\le& \int_{\{f(t_1)<v<f(t_2), |\nabla v|>0\}}
\lf(\frac{(2-p)^2|\nabla v|^3}{|\nabla v|_\e^4 }
-\frac{(2-p)|\nabla v|}{|\nabla v|_\e^2}\ri)v_{\nu\nu}^2\\
&\le (2-p)(1-p) \int_{\{f(t_1)<v<f(t_2), |\nabla v|>0\}}  \frac{|\nabla v|}{|\nabla v|_\e^2 }v_{\nu\nu}^2
\end{split}
\eee
for some $C_3$ independent of $\e$. Hence
\bee
\int_{\{f(t_1)<v<f(t_2), |\nabla v|>0\}}  \frac{|\nabla v|}{|\nabla v|_\e^2 }v_{\nu\nu}^2\le (2-p)^{-1}(p-1)^{-1}C_3,
\eee
and
\bee
\begin{split}
&\int_{\{f(t_1)<v<f(t_2), |\nabla v|>0\}}\frac{|\Delta v|}{|\nabla v|}\\=&\int_{\{f(t_1)<v<f(t_2), |\nabla v|>0\}} (2-p)\frac{|\nabla v|}{|\nabla v|_\e^2}|v_{\nu\nu}|\\
\le &(2-p)\e^{-1}\int_{\{f(t_1)<v<f(t_2), |\nabla v|>0\}} \frac{|\nabla v|}{|\nabla v|_\e}|v_{\nu\nu}|\\
\le &\frac12(2-p)\e^{-1}\int_{\{f(t_1)<v<f(t_2), |\nabla v|>0\}} \frac{|\nabla v|}{|\nabla v|_\e} \lf(\frac1{|\nabla v|_\e}v_{\nu\nu}^2+|\nabla v|_\e\ri)\\
\le&C\e^{-1}
\end{split}
\eee
for some constant $C>0$ independent of $\e$. Here we have used the fact that $|\nabla v|_\e\ge\e$, and $|\nabla v|\to |\nabla u|$.
\end{proof}

We are ready to  prove Theorem \ref{t-intro-1}. Before we prove the theorem, let us fix some notation. Let $u$ be as in the theorem, then $|\nabla u|>0$ outside a compact set by the asymptotic behavior \eqref{e-behavior} of $u$. Let $0<t_1<t_2$ be such that such that $\Sigma(t_1)=\{u=f(t_1)\}, \Sigma(t_2)=\{u=f(t_2)\}$ are regular. Recall that $f(t)=1-ct^{-a}$, $a=(3-p)/(p-1)$, $1<p\le 2$.  Fix $T>t_2$ so that $\Sigma(T)$ is regular. For any $\e>0$, let $v_\e$ be the solution of \eqref{e-v-1}.   Let $f_\e(t)=1-c_\e t^{-a}$ and $c_\e $ be as in \eqref{e-capacity-1}. Let $D_\e(t)$ be as in \eqref{e-De} whenever $\Sigma(\e,t)=\{v_\e=f_\e(t)\}$ is regular. By \cite{AMMO2022}, for $\e>0$ small enough,   $\Sigma(\e,t_1), \Sigma(\e,t_2)$ are regular.

\begin{proof}[Proof of Theorem \ref{t-intro-1}] (i) With the above setting,  by Lemma \ref{l-monotone-d}, for $\e$ small enough, we have
 \bee
D_\e(t_2)-D_\e(t_1) \le \int_{\{f_\e(t_1)<v<f_\e(t_2)\}}E(\e).
\eee
where $E(\e)\ge0$ is uniformly bounded independent of $\e$ and converges to zero everywhere in $M(T)=\{0<u<f(T)\}$.
Let $\e\to0$, we conclude by Lemma \ref{l-approx}
$$
D(t_2)-D(t_1)\le 0.
$$

(ii) If $M$ is AF, then $\Sigma(t_2)$  is regular   for   $t_2>>1$ and $D(t_2)\to 0$ as $t_2\to\infty$ by the proof of Proposition \ref{t-section-2}. By (i) we have $D(t)\ge 0$
whenever $\Sigma(t)$ is regular. Since
$$ D =(2a+1) \cB- a \cA, $$
we conclude that (ii) is true.

To prove (iii), recall
$$
\cB_\e(t)=4\pi t-(c_\e a)^{-2}t^{2a+1}\int_{\Sigma(\e,t)}|\nabla v_\e|^2
$$
whenever $\Sigma(\e,t)$ is regular.
\bee
\begin{split}
\cB_\e(t)
=&c^\frac1a_\e\int_{\Sigma(\e, t)}(c_\e a)^{1-p} (1-v_\e)^{-\frac1a} |\nabla v_\e|_\e^{p-2}|\nabla v_\e|- a^{-2} (1-v_\e)^{-(2+\frac1a)}  |\nabla v_\e|^2\\
=&c^\frac1a\int_{\Sigma(\e,t)}Q_\e
\end{split}
\eee
Try to find $X_\e$ so that at $\Sigma(\e,t)$, $\la X_\e,\frac{\nabla v_\e}{|\nabla v|_\e}\ra=Q_\e.$ So let
\be\label{e-X}
X_\e=(c_\e a)^{1-p} (1-v_\e)^{-\frac1a} |\nabla v_\e|_\e^{p-2}\nabla v_\e-a^{-2}  (1-v_\e)^{-(2+\frac1a)} |\nabla v_\e|\,\nabla v_\e=:Z_\e-Y_\e.
\ee

For any $\delta>0$, let
$$Y_{\e,\delta}=a^{-2}  (1-v_\e)^{-(2+\frac1a)} \sqrt{|\nabla v_\e|^2+\delta^2}\nabla v_\e.$$
Then $Z_\e-Y_{\delta,\e}$ is a smooth vector field. Since $\Sigma(t_1), \Sigma(t_2)$ are regular, we can conclude that $\Sigma(\e,t_1), \Sigma(\e,t_2)$ are regular for small $\e$ and so:

\be\label{e-Be-1}
c^\frac1a_\e\lim_{\delta\to 0}\int_{\{f_\e(t_1)<v_\e<f_\e(t_2)\}}\div(Z_\e-Y_{\delta,\e})= \cB_\e(t_2) - \cB_\e(t_1).
\ee

Now
\be\label{e-div-Z}
\begin{split}
\div \, Z_\e
=&(c_\e a)^{1-p}\frac1a(1-v_\e)^{-(\frac1a+1)}|\nabla v_\e|_\e^{p-2}|\nabla v_\e|^2,
\end{split}
\ee
which is zero if $|\nabla v_\e|=0$.
On the other hand, if $|\nabla v_\e|=0$, then
\bee
\div\, Y_{\delta,\e}=a^{-2} \delta  (1-v_\e)^{-(2+\frac1a)} \Delta v_\e.
\eee
If $|\nabla v_\e|>0$, then

\bee
\begin{split}
\div\, Y_{\delta,\e}=&a^{-2}   (1-v_\e)^{-(2+\frac1a)}\bigg[ \sqrt{|\nabla v_\e|^2+\delta^2}\Delta v_\e+  \frac{|\nabla v_\e|\la\nabla |\nabla v_\e| ,\nabla v_\e\ra}{\sqrt{|\nabla v_\e|^2+\delta^2}}\\
&+ (2+\frac1a)(1-v_\e)^{-1}\sqrt{|\nabla v_\e|^2+\delta^2}|\nabla v_\e|^2\bigg]\\
\end{split}
\eee
Note that for $|\nabla v_\e|> 0$, $\la\nabla |\nabla v_\e| ,\nabla v_\e\ra=|\nabla v_\e|(v_\e)_{\nu\nu}$.
Since $v_\e$ is smooth, one can see that $\div Z_\e$ is bounded and $\div\, Y_{\delta,\e}$ is uniformly bounded independent of $\delta$. Moreover, $\div\, Y_{\delta,\e}\to \div\, Y_\e$ if $|\nabla v_\e|> 0$ as $\delta\to0$ and $\div\, Y_{\delta,\e}\to 0$ if $|\nabla v_\e|=0$. By Lebesgue's dominated convergence theorem, let $\delta\to0$ we have:
\bee
c^{-\frac1a}( \cB_\e(t_2) - \cB_\e(t_1))= \int_{\{f_\e(t_1)<v_\e<f_\e(t_2)\}}\mathbf{1}_{\{|\nabla v_\e|>0\}} \div(Z_\e-Y_\e).
\eee
At $|\nabla v_\e|>0$,
\be\label{e-ZY}
\begin{split}
&|\nabla v_\e|^{-1}\div(Z_\e-Y_\e)\\
=&a^{-1}(1-v_\e)^{-(\frac1a+2)}\bigg[  \frac{4\pi}{C_{p,\e}}(1-v_\e)|\nabla v_\e|_\e^{p-2}|\nabla v_\e|
-a^{-1}\lf( \Delta v_\e+(v_\e)_{\nu\nu} \ri)
\\&- a^{-2}(1+2a) (1-v_\e)^{-1}  |\nabla v_\e|^2\bigg].
\end{split}
\ee
Since
$1-v_\e\ge C>0$ for $f_\e(t_1) <v_\e<f_\e(t_1)<1$   for some $C>0$,
and $v_\e$ is smooth,
  $|\nabla v_\e|^{-1}{\bf 1}_{\{|\nabla v_\e|>0\}}\div(Z_\e-Y_\e)$ is in $L^1$. Let $\tau=f_\e(t)$, $\tau_1=f_\e(t_1), \tau_2=f_\e(t_2)$ and let    $R$ be the set of regular values of $v_\e$. By the co-area formula and the Morse-Sard's theorem, using \eqref{e-De}, \eqref{e-ZY} and the fact that $\tau=f_\e(t)$,

\bee
\begin{split}
 &c^{-\frac1a}_\e( \cB_\e(t_2) - \cB_\e(t_1))\\
 =& \int_{\{f_\e(t_1)<v_\e<f_\e(t_2)\}} {\bf 1}_{\{|\nabla v_\e|>0\})}\div(Z_\e-Y_\e){\bf 1}_{\{|\nabla v_\e|>0\})}\\
 =&\int_{\tau_1}^{\tau_2}\int_{\{v_\e=\tau\}} |\nabla v_\e|^{-1}{\bf 1}_{\{|\nabla v_\e|>0\})}\div(Z_\e-Y_\e) d\tau\\
 =&\int_{\{\tau_1<\tau<\tau_2, \tau\in R\}}\int_{\{v_\e=\tau\}} |\nabla v_\e|^{-1}\div(Z_\e-Y_\e) d\tau\\
 =&c_\e a^{-1}\int_{\{\tau_1<\tau<\tau_2, \tau\in R\}}(1-\tau)^{-\lf(\frac1a+2\ri)}\\
 &\cdot\bigg[c_\e D_\e (t(\tau))-a^{-1}\lf( \Delta v_\e+(v_\e)_{\nu\nu} \ri)-\Delta v_\e-(v_\e)_{\nu\nu}\bigg] d\tau.
\end{split}
\eee
On the other hand,
\bee
\begin{split}
 a^{-1}&\lf( \Delta v_\e+(v_\e)_{\nu\nu} \ri)+\Delta v_\e-(v_\e)_{\nu\nu}\\
  =&\lf[\frac{p-1}{3-p}\frac{(3-p)|\nabla v_\e|^2+\e^2}{|\nabla v_\e|_\e^2}+\frac{(1-p)|\nabla v|^2-\e^2}{|\nabla v_\e|_\e^2}\ri](v_\e)_{\nu\nu}\\
  =&\frac{2(p-2)\e^2}{3-p}\frac{\Delta v_\e}{|\nabla v_\e|^2}.
\end{split}
\eee
Hence
\bee
\begin{split}
&c^{-\frac1a}( \cB_\e(t_2) - \cB_\e(t_1))\\
=&\int_{\tau_1<\tau<\tau_2, \tau\in R}a^{-1}c_\e(1-\tau)^{-2-\frac1a}D_\e(\tau) d\tau
\\
&+\int_{\{\tau_1<\tau<\tau_2, \tau\in R\}}a^{-1}(1-\tau)^{-2-\frac1a}\int_{\Sigma_{\e,\tau}} \frac{2(p-2)\e^2}{3-p}\frac{\Delta v_\e}{|\nabla v_\e|^2} d\tau \\
=&\int_{\{\tau_1<\tau<\tau_2, \tau\in R\}}a^{-1}c_\e(1-\tau)^{-2-\frac1a}D_\e(\tau) d\tau
\\
&+\int_{\{\tau_1<v_\e<\tau_2\}} a^{-1} \mathbf{1}_{\{|\nabla v_\e|>0\}}(1-v_\e)^{-2-\frac1a} \frac{2(p-2)\e^2}{3-p}\frac{\Delta v_\e}{|\nabla v_\e|} d\tau\\
=&(1)+(2).
\end{split}
\eee
By Lemma \ref{l-monotone-d}(ii),
$$
(2)\ge -C_1\e
$$
for some constant $C_1>0$ independent of $\e$. By Lemma \ref{l-monotone-d}(i),
\bee
(1)\ge C_2 D_\e(t_2)+o(1)
\eee
as $\e\to 0$ for some constant $C_2>0$ independent of $\e$. Let $\e\to 0$, we conclude that

as $\e\to0$. Let $\e\to0$ by Lemma \ref{l-approx} we have
$$
c^{-\frac1a}( \cB(t_2) - \cB(t_1))\ge C_2D(t_2)\ge0,
$$
by part (ii).

In \cite{AMMO2022}, it was shown that $ F( t_2)\ge F(t_1)  $ . Since $F$, $\cA$ and $\cB$ are related by
$F (t) = \cA (t) - \cB(t)$,
the monotonicity of $\cA (t_2)\ge \cA(t_1)$.

This finishes the proof of Theorem  \ref{t-intro-1}.
\end{proof}

\section{Asymptotical behavior}\label{s-asy}
Let $(M^3,g)$, $\cA(t) $ and $ \cB(t)$ be given as in Theorem \ref{t-intro-1}. By \eqref{e-behavior}, the level set
$\Sigma(t) = \{ u = f (t) \}$ is regular when $t$ is large. By Theorem \ref{t-intro-1}, $\cA(t)$ and $ \cB(t)$ are nondecreasing in $t$ for those $t$ with $\Sigma(t)$ being regular.

If $p=2$,  it is known, as $x \to \infty$,  the harmonic function $u$ has an asymptotic expansion
\be \label{eq-u-1}
u = 1 -  c | x |^{-1} + O_2 ( | x |^{-1 - \sigma} ) ,
\ee
see \cite[Lemma A.2]{MMT18} for instance.
Here $ c $ is a positive constant and $ \sigma \in (\frac12, 1)$ is a decay rate of $g_{ij}$ in \eqref{e-AF-1}.
With the help of \eqref{eq-u-1}, it was shown in \cite{Miao2022} that $\cA(t)$ converges to $12\pi \mathfrak{m}$ and
$ \cB(t)$ converges to $4\pi \mathfrak{m}$, respectively, as $t\to \infty$.
(Note that $\cA $ and $\cB$ here differ with those in \cite{Miao2022} by a  factor of $\frac{1}{4\pi} C_2$.)

If $ p \in (1,3) $ and $p \ne 2$, unlike harmonic functions,  near infinity we only have
\be\label{e-u-asy}
u=1-cr^{-a}+o_2(r^{-a})
\ee
(to the authors' knowledge).
 {Note that this in particular implies that level-sets of $u$ are regular near $\infty$.}
Nevertheless, a Hawking mass estimate
\be \label{eq-Hmass-limit}
\limsup_{t \to \infty } \frac{t}{2} \left( 1 - \frac{1}{16 \pi} \int_{\Sigma(t)} H^2 \right) \le \m
\ee
was proved in  \cite[Lemma 2.5]{AMMO2022}.
\eqref{eq-Hmass-limit} can be viewed as a Hawking mass estimate because
the ratio between $t$ and the area-radius of $\{ \Sigma(t) \}$ tends to $1$ by \eqref{e-u-asy}.
From this, it was shown in \cite{AMMO2022} that
$$\limsup_{t\to\infty} F(t) \le 8\pi \mathfrak{m} . $$

As a corollary of \eqref{eq-Hmass-limit} proved in \cite{AMMO2022}, one has

\begin{prop}\label{t-asy}
\begin{enumerate}\
  \item [(i)] $\limsup_{t\to\infty} \cA(t)\le 4\pi(5-p)\mathfrak{m}$.
  \item [(ii)] $\limsup_{t\to\infty} \cB(t)\le 4\pi(3-p)\mathfrak{m}$.
\end{enumerate}
\end{prop}
\begin{proof} Let $c$ be as in the proof of Theorem \ref{t-intro-1}. Namely, for $1<p<3$,
$$
ca=\lf(\frac{C_p}{4\pi}\ri)^\frac1{p-1}
$$
where $a=(3-p)/(p-1)$ and $C_p$ is the $p$-capacity of $\p M$.  Let $\tau=f(t)$. Then
\bee
t=(\frac c{1-\tau})^\frac1a.
\eee
Let $\Sigma_\tau=\{u=\tau\}$. Suppose $\tau$ is a regular value of $u$. By Lemma \ref{l-basic-2},
\be \label{eq-dguH}
\begin{split}
\frac{d}{d\tau}\int_{\Sigma_\tau}|\nabla u|H=& \int_{\Sigma_\tau} \lf( K- \frac 3{4}H^2 +H\frac{\Delta u}{|\nabla u|}\ri) -E(\tau)\\
\le & 4\pi -\int_{\Sigma_\tau} \frac{5-p}{4(p-1)}H^2
\end{split}
\ee
where
\bee
E(\tau)=\int_{\Sigma_\tau} \lf(\frac{|\wt\nabla |\nabla u|^2|}{|\nabla u|^2} +\frac12(\cS+|\mathring A|^2)\ri) \ge 0  .
\eee

Given any $\tilde \m > \m$, by \eqref{eq-Hmass-limit},
\be \label{eq-mhmass-bd}
-\int_{\Sigma_\tau}H^2\le 32\pi c^{-\frac1a} (1-\tau)^{\frac1a}  \tilde {\m} -16\pi
\ee
for $\tau$ close to $1$.
By \eqref{e-u-asy},  $\int_{\Sigma_\tau}|\nabla u|H\to 0$ as  $\tau \to 1$.
Integrating  \eqref{eq-dguH} and using \eqref{eq-mhmass-bd}, we have
\be\label{e-uH}
\begin{split}
-\int_{\Sigma_\tau}|\nabla u|H \le & \  -8\pi a(1- \tau) \\
& \ +8\pi \cdot \frac{5-p}{p-1} c^{-\frac1a}\frac a{1+a}(1-\tau)^{1+\frac1a} \tilde \m .
\end{split}
\ee
Therefore,
\bee
\begin{split}
\cA(t)=  \cA(t(\tau))
=& \ (\frac c{1-\tau})^\frac1a \lf(8\pi-  \frac 1{a(1-\tau)}\int_{\Sigma_\tau}|\nabla u|H \ri)  \\
\le & \  8\pi \frac{5-p}{p-1} \frac 1{1+a} \tilde \m \\
= & \ 4\pi(5-p) \tilde \m .
\end{split}
\eee
As $\tilde \m > \m $ is arbitrary, from this (i) follows.

To show (ii), by \eqref{e-uH},
\bee
\begin{split}
\frac{d}{d\tau}\int_{\Sigma_\tau}|\nabla u|^2=& -a\int_{u=\tau}|\nabla u|H \\
\le &a\bigg[-8\pi a(1- \tau)+8\pi \cdot \frac{5-p}{p-1} c^{-\frac1a}\frac a{1+a}(1-\tau)^{1+\frac1a}
\tilde \m  \bigg] .
\end{split}
\eee
By \eqref{e-u-asy} and the fact $|\Sigma(t)|=O(t^2)$, we have $\int_{\Sigma_\tau} |\nabla u|^2\to0$ as $\tau\to1$. Hence,
\bee
\begin{split}
-\int_{\Sigma_\tau}|\nabla u|^2\le & \ -4\pi a^2(1-\tau)^2 \\
& \
+8\pi \cdot \frac{5-p}{p-1} c^{-\frac1a} \frac {a^3}{(1+a)(1+2a)}  (1-\tau)^{2+\frac1a} \tilde \m .
\end{split}
\eee
This implies
\bee
\begin{split}
\cB(t) = \cB(t(\tau))
= & \ (\frac{c}{1-\tau})^{\frac1a}\lf(4\pi-a^{-2}(1-\tau)^{-2}\int_{\Sigma_\tau}|\nabla u|^2 \ri)\\
 \le & \ 4\pi(3-p) \tilde \m .
\end{split}
\eee
From this (ii) follows.
\end{proof}

\section{Applications}

We give applications of results in the previous sections. First recall the definitions of some quantities. Consider a complete, orientable, asymptotically flat $3$-manifold  $(M^3,g)$ with smooth boundary $\p M$. For $1<p\le 2$, denote

\be\label{e-quantities}
\left\{
  \begin{array}{ll}
    C_p=\text{$p$-capacity of $\p M$ in $(M,g)$}; \\
    a = \frac{3-p}{p-1}\\
    c   =a^{-1} \left( \frac{C_p}{ 4\pi} \right)^\frac{1}{p-1}
  \end{array}
\right.
\ee

By Theorem \ref{t-intro-1} and Proposition \ref{t-asy}, we have:

\begin{thm}\label{t-ratio}
Let $(M^3,g)$ be a complete, orientable, asymptotically flat $3$-manifold with smooth boundary $\p M$.
Suppose $\p M$ is connected and $H_2(M,\p M)= 0$.
For $1<p\le 2$, let $u$ be the $p$-harmonic function on $  M $ with $u=0$ on $\p M $, and $u\to 1$ at infinity.
If $g$ has nonnegative scalar curvature, then
\bee
( 1 - s ) 4 \pi  +   \int_{ u = s }   | \nabla u | H -
a^{-2} ( 1 + 2a)  ( 1 - s )^{- 1} \int_{ u = s   } | \nabla u |^2   \ge 0 ,
\eee
\bee
c^{\frac1a} (1- s )^{-\frac1a} \lf(8\pi-  \frac 1{a(1-s )}\int_{ u = s  } |\nabla u|H \ri)\le 4\pi (5-p)\mathfrak{m},
\eee
and
\bee
c^{\frac1a} ( 1- s)^{-\frac1a}\lf(4\pi-a^{-2}(1-s )^{-2}\int_{ u = s }|\nabla u|^2 \ri)
\le 4\pi (3-p)\mathfrak{m},
\eee
whenever $s$ is a regular value of $u$. In particular, at $\p M$,
\be \label{ineq-hu-gdu}
4 \pi  +  \int_{\p M }   | \nabla u | H \ge a^{-2} ( 1 + 2a) \int_{\p M} | \nabla u |^2   ,
\ee
\be \label{ineq-Hup}
c^{\frac1a}  \left(  8\pi- a^{-1} \int_{\p M}|\nabla u|H \right) \le 4\pi (5-p)\mathfrak{m},
\ee
and
\be \label{ineq-dup}
c^{-\frac1a} \left(  4\pi-a^{-2} \int_{\p M}|\nabla u|^2 \right)
\le 4\pi (3-p)\mathfrak{m}.
\ee
Here $H$ is the mean curvature of $\p M $ with respect to $\nu=\nabla u/|\nabla u|$.
Moreover, if equality holds in any of \eqref{ineq-hu-gdu} -- \eqref{ineq-dup}, then $(M,g)$ is isometric to $\R^3$ outside a round ball.
\end{thm}

If $ p = 2$, Theorem \ref{t-ratio} reduces to inequalities in \cite[ Theorems 3.1, 3.2]{Miao2022}.
Similar to the $p=2 $ case, we note \eqref{ineq-hu-gdu} $+$ \eqref{ineq-dup} $ \Longrightarrow $ \eqref{ineq-Hup}.

If $H=0$, \eqref{ineq-Hup} reduces to $ 2  c^{\frac1a}(5-p)^{-1}   \le  \mathfrak{m}$, i.e.
\be \label{ineq-Hup-H-0}
\frac{2}{5-p}  \left( \frac{ 3 - p }{ p - 1 } \right)^\frac{ 1 -p}{3 - p} \left( \frac{ C_p}{ 4 \pi} \right)^\frac{1}{3-p} \le \m .
\ee
Similar to the applications  in \cite{AMMO2022}, by \cite[Theorem 1.2]{FM22} letting $p \to 1$ in \eqref{ineq-Hup-H-0}, one retrieves
the Riemannian Penrose inequality in the case that $\p M$ is a connected, outer minimizing  surface.

Next, we give a corollary of  \eqref{ineq-hu-gdu} and \eqref{ineq-dup}.

\begin{cor} \label{cor-mass-pc-Willmore}
Let $(M^3,g)$ be a complete, orientable, asymptotically flat $3$-manifold with smooth boundary $\p M$.
Suppose $\p M$ is connected and $H_2(M,\p M)= 0$.
Let $ W = \frac{1}{16 \pi} \int_{\p M} H^2 $, where $H$ is the mean curvature of $\p M$.
For $1<p\le 2$, if  $g$ has nonnegative scalar curvature, then
\be \label{eq-mass-cp-H}
1 \le a^{ \frac{1}{a} }   \left( \frac{ 4 \pi }{C_p} \right)^\frac{1}{3-p} (3-p) \mathfrak{m}
+ \frac{a^2}{ ( 1 + 2 a)^2} \left( \sqrt{W} + \sqrt{ W + \frac{ 1 + 2 a}{a^2}  } \right)^2 .
\ee
If equality holds, then $(M, g)$ is isometric to $\R^3$ minus a round ball.

As a result of \eqref{eq-mass-cp-H}, if $ \p M$ is area outer-minimizing in $(M,g)$, then
\be \label{eq-g-Hawkingmass-m}
\sqrt{ \frac{ | \p M | }{ 16 \pi } }  \left(
1 - \frac{1}{16 \pi} \int_{\p M} H^2 \right) \le   \mathfrak{m} .
\ee
\end{cor}

\begin{proof}
Let $u$ be the $p$-harmonic function on $  M $ with $u=0$ on $\p M $ and $u\to 1$ at infinity.
By \eqref{ineq-hu-gdu} and H\"{o}lder's inequality,
\bee
4 \pi  +  \sqrt{ 16 \pi W}  \left( \int_{\p M }   | \nabla u |^2  \right)^\frac12  \ge a^{-2} ( 1 + 2a) \int_{\p M} | \nabla u |^2  .
\eee
This implies
\be \label{eq-bound-int-gdu-2}
a^{-2} \int_{\p M }   | \nabla u |^2 \le
4 \pi \frac{a^2}{ ( 1 + 2 a)^2} \left( \sqrt{W} + \sqrt{ W + \frac{ 1 + 2 a}{a^2}  } \right)^2
\ee
by elementary reasons.
\eqref{eq-mass-cp-H} follows from \eqref{ineq-dup} and  \eqref{eq-bound-int-gdu-2}.

Letting $ p \to 1$ in \eqref{eq-mass-cp-H} and using the fact
 $ \{ C_p \}_{1 < p \le 2 }$ is bounded (which can be seen by
choosing a fixed test function in the variational definition of $C_p$),
we have
\be \label{eq-capacity-Hawking-mass}
\sqrt{ \frac{  \limsup_{ p \to 1} C_p}{ 16 \pi } }  \left(
1 - \frac{1}{16 \pi} \int_{\p M} H^2 \right) \le   \mathfrak{m} .
\ee
If $\p M$ is area outer-minimizing in $(M, g)$, it was shown in \cite{FM22} (also see \cite{AFM22, AMMO2022})
that $ \lim_{p \to 1} C_p  = | \p M | $.
Hence, \eqref{eq-g-Hawkingmass-m} follows from \eqref{eq-capacity-Hawking-mass}.
\end{proof}

\begin{rem}
For each fixed $ p$, \eqref{eq-mass-cp-H} implies the $3$-dimensional Riemannian positive mass theorem.
For instance, suppose $M$ is topologically $ \R^3$ and apply \eqref{eq-mass-cp-H} to the complement of a small
geodesic ball $B_r(x)$ with radius $r$ in $(M,g)$. Since $ C_p$ remains bounded and $W \to 1$, as $ r \to 0$, we see $ \m \ge 0$.
\end{rem}

{
\begin{rem}
In \cite{Moser2007}, Moser gave another proof of the
existence of weak inverse mean curvature flow (IMCF) by constructing $p$-harmonic functions for $p>1$ and letting $p\to 1$.
In \cite{Xiao2016}, Xiao used IMCF to obtain estimates on $p$-capacity
and commented on their limit as $p \to 1$.
We do not use IMCF in this work.
Our approach is  along the same line as in \cite{AMMO2022}. We would like to thank Jie Xiao for bring our
attention to Remark 1.2 in the work [41].
\end{rem}
}

\begin{thm} \label{thm-sec-pmt-bdry}
Let $(M, g)$ be a complete, orientable, asymptotically flat $3$-manifold
with nonnegative scalar curvature, with smooth boundary $\p M $.
Suppose $ \p M $ is connected and $H_2 (M, \p M ) = 0$.
Let $ H_{\max} $ denote the maximum of the mean curvature $H$ of $\p M$.
Suppose $ H_{\max} \ge 0 $.
Then for $1<p\le 2$,
\be \label{eq-mass-H-c}
\begin{split}
2 \le & \      a^{\frac{1}{a}} \left( \frac{ 4 \pi }{ C_p} \right)^\frac{1}{3- p}        (5-p)\mathfrak{m} \\
 & \ +  H_{\max}  a^{ \frac{1-p}{3-p} } \left( \frac{ C_p}{ 4 \pi} \right)^\frac{1}{3-p}
 \left[  \frac{  a \left( \sqrt{W} + \sqrt{ W + \frac{ 1 + 2 a}{a^2}  } \right) }{ ( 1 + 2 a)} \right]^\frac{ 2 ( 2-p) }{3-p} .
\end{split}
\ee
Consequently,  if $ \Omega \subset M$ is a bounded region separating $\p M $
and $\infty$, meaning that
$ \p \Omega $ has two connected components $S_0$ and $S_1$, where
$S_0$ encloses $\p M $ (and is allowed to
coincide with $\p M $) and $S_1$ encloses $S_0$, then
$$ H_{\max}
 \left[  \frac{  a \left( \sqrt{W} + \sqrt{ W + \frac{ 1 + 2 a}{a^2}  } \right) }{ ( 1 + 2 a)} \right]^\frac{ 2 ( 2-p ) }{3-p}
 \le 2 \left( \frac{ 4 \pi }{ C_p (\Omega) }  \right)^\frac{1}{3-p}  a^{ \frac{p-1}{3-p} } \
 \Longrightarrow \m > 0 . $$
\end{thm}

\begin{proof}
Since $ p \le 2$,
\bee
\begin{split}
\int_{\p M} | \nabla u | = & \  \int_{\p M} | \nabla u |^\frac{p-1}{3-p}  | \nabla u |^{ \frac{ 2 ( 2 -p)}{ 3 - p} } \\
\le & \ \left(  \int_{\p M} | \nabla u |^{p-1} \right)^\frac{1}{3-p} \left( \int_{\p M} | \nabla u |^2 \right)^\frac{ 2-p}{3-p} .
\end{split}
\eee
Recall the $p$-capacity of $\Sigma$ in $(M, g)$ is given by
$$ C_p = \int_{\p M  } | \nabla u |^{p-1} . $$
Thus,
\be \label{eq-est-gdu-l-1}
\int_{\p M} | \nabla u |
\le C_p^\frac{1}{3-p} \left( \int_{\p M} | \nabla u |^2 \right)^\frac{ 2-p}{3-p} .
\ee
As $ H_{\max} \ge 0 $, \eqref{eq-mass-H-c} follows from \eqref{ineq-Hup}, \eqref{eq-est-gdu-l-1} and \eqref{eq-bound-int-gdu-2}.

Suppose   $ \Omega \subset M$ is a bounded region separating $\p M $
and $\infty$. Let $u_{_\Omega}$ be the $p$-harmonic function with
$  u_{_\Omega} |_{S_0} = 0 $ and  $u_{_\Omega} |_{S_1 } = 1 . $
Let
$$  C_p (\Omega)  =  \int_{\Omega} | \nabla u_{_\Omega}  |^{p}
= \int_{S_0}  | \nabla u_{_\Omega} |^{p-1} . $$
Then $ C_p < C_p (\Omega)$ by the variational nature of the $p$-capacity.
Hence, \eqref{eq-mass-H-c} holds with $C_p$ replaced by $C_p (\Omega)$ and the inequality becomes strict.
This implies the rest claim.
\end{proof}

\begin{rem}
One can have a rough estimate of $C_p (\Omega)$ in terms of $\Vol(\Omega)$, the volume of $(\Omega, g)$,
and $L$, the distance between $S_0$ and $S_1$.
Let $f(x)$ be a test function so that $f=0$ on the region enclosed by $ S_0$ with $\p M $,
$f(x) = L^{-1} d (x) $ for $x $ outside $S_0$ with $ d (x) \le L$ and $f(x) = 1$ if $ d (x) \ge L$.
Here $d(x)$ denote the distance from $x$ to $S_0$.
Then
$$
 C_p (\Omega) \le \int_\Omega | \nabla f |^p \le  L^{-p} \,  \Vol (\Omega) .
$$
Thus, Theorem \ref{thm-sec-pmt-bdry} implies the following: If
\be \label{eq-localized-condition}
 H_{\max}
 \left[  \frac{  a \left( \sqrt{W} + \sqrt{ W + \frac{ 1 + 2 a}{a^2}  } \right) }{ ( 1 + 2 a)} \right]^\frac{ 2 ( 2-p ) }{3-p}
 \le 2 \left( \frac{ 4 \pi L^p }{ \Vol ( \Omega) }  \right)^\frac{1}{3-p}  a^{ \frac{p-1}{3-p} } ,
\ee
then $  \m > 0 $.
Note the right side of \eqref{eq-localized-condition} does not depend on $\p M$.
If $ p = 2$, \eqref{eq-localized-condition} becomes $ H_{\max} \le 8 \pi L^2 {\Vol (\Omega)}^{-1} $,
which is exactly the condition in \cite[Equation (5.2)]{Miao2022}.
\end{rem}

Recently, there are results on positive mass theorems on complete manifolds with arbitrary ends,
see \cite{LesourdUngerYau,LeeLesourdUger,Zhu} for instances. We apply Theorem \ref{t-ratio} to obtain a special case of
those results.

\begin{prop}\label{p-pmt}
Let $(M^3, g)$ be a complete noncompact manifold with nonnegative scalar curvature with two ends $E, \wt E$ such that $M$ is diffeomorphic to $\R^3\setminus\{0\}$ and such that $E$ is AF. Suppose there is a harmonic function $u$ such that $u\to 1$ at the infinity of $E$ and $u\to 0$ at the infinity of $\wt E$ and suppose the Ricci curvature of $M$ is bounded below. Then the ADM mass of $E$ satisfies:
$$
8\pi\m\ge C_2>0
$$
where $C_2$ is the capacity of $u$. That is $C_2=\int_{\{u=\tau\}}|\nabla u|$, where $0<\tau<1$ is any constant so that $\{u=\tau\}$ is regular.
\end{prop}
\begin{proof} Assume that the origin $0$ corresponds to the infinity of $\wt E$. Let $\rho_i\downarrow0$ and let $v_i$ be the harmonic function such that $v_i=0$ at $\p B_0(\rho_i)$ and $v_i\to 1$ at infinity of $E$. Here $B_0(\rho)$ is the Euclidean ball of radius $\rho$. Then $v_i$ will converge to $u$ in $C^\infty$ norm on compact sets.

Fix $0<\tau<1$ so that $\tau$ is a regular value of $u$. Hence if $i$ is large enough, then $\tau$ is also a regular value of $v_i$.
By \cite{Miao2022} or Theorem \ref{t-ratio}  with $p=2$, we have
\bee
\begin{split}
8\pi \m\ge& \frac{C_2(i)}{4\pi}(1-\tau)^{-1}\lf( 4\pi-(1-\tau)^{-2}\int_{\{v_i=\tau\}}|\nabla v_i|^2\ri).\\
\end{split}
\eee
Here $C_2(i)$ is the $2$-capacity of $v_i$. Since $v_i\to u$, by the maximum principle, there is a compact set $Q$ in $M$ such that
$B_x(1)\subset Q$ for all $x\in \{v_i=\tau\}$. Here $B_x(1)$ is the geodesic ball centered as $x$ with  radius 1.  Since the Ricci curvature of $M$ is bounded from below, by the gradient estimate for positive harmonic functions by  Cheng-Yau \cite{ChengYau}, see also \cite[Theorem 6.1]{Li}, we conclude that there is a constant $\beta$ depending only on the lower bound of the Ricci curvature so that
$$
|\nabla v_i(x)|\le \beta v_i(x)=\beta\tau
$$
for all $x\in \{v_i=\tau\}$ for all $i$, if $i$ is large enough.
Hence we have:
\bee
\begin{split}
8\pi \m\ge& \frac{C_2(i)}{4\pi}(1-\tau)^{-1}\lf( 4\pi-\b\tau(1-\tau)^{-2}\int_{\{v_i=\tau\}}|\nabla v_i|\ri) \\
=&\frac{C_2(i)}{4\pi}(1-\tau)^{-1}\lf( 4\pi-\b C_2(i)\tau(1-\tau)^{-2} \ri).
\end{split}
\eee
Let $i\to\infty$, since $C_2(i)\to C_2$, we have
\bee
\begin{split}
8\pi \m\ge&\frac{C_2(i)}{4\pi}(1-\tau)^{-1}\lf( 4\pi-\b C_2 \tau(1-\tau)^{-2} \ri).
\end{split}
\eee
Let $\tau\to 0$, the result follows.
\end{proof}

\appendix
\section{Basic computations}\label{s-prelim}

Let $(M^n,g)$ be a Riemannian manifold and $u$ is a smooth function on $M$. Let $\Sigma_\tau=\{u=\tau\}$. Assume $|\nabla u|>0$ on $\Sigma_\tau$. We have the following:

\begin{lma}\label{l-basic-1}   Let $\nu=\frac{\nabla u}{|\nabla u|}$. Let $H$ be the mean curvature of $\Sigma_\tau$ with respect to $\nu$. Then

\bee
\left\{
  \begin{array}{ll}
    \frac{\p}{\p \tau}=\frac1{|\nabla u|}\nu=\frac{\nabla u}{|\nabla u|^2};\\
  H=\frac1{|\nabla u|}\lf(\Delta u-u_{\nu\nu} \ri);\\
\frac{\p}{\p \tau}|\nabla u|^2=2\lf(\Delta u-H|\nabla u|\ri);\\
\frac{\p}{\p \tau}d\sigma=\frac1{|\nabla u|}H d\sigma;\\
|\nabla u|\frac{\p}{\p \tau}H=-\frac12(\frac n{n-1}H^2-\cS^\tau)-\lf(|\nabla u|\wt\Delta \frac{1}{|\nabla u|}+\frac12(\cS+|\mathring A|^2)\ri);
\end{array}
\right.
\eee
where $u_{\nu\nu}=\nabla^2 u(\nu,\nu)$, $\nu=\nabla u/|\nabla u|$; $d\sigma$ is the area element of $\Sigma_\tau$; $\cS$ is the scalar curvature of $M$; $\cS^\tau$ is the scalar curvature of $\Sigma_\tau$; $\mathring A$ is the traceless part of the second fundamental form $A$ of $\Sigma_\tau$ with respect to $\nu$; and $\wt \Delta$ is the Laplacian operator of $\Sigma_\tau$ with respect to the induced metric.

Moreover, if $X, Y$ are tangential to $\Sigma_\tau$, then $\nabla^2u(X,Y)=|\nabla u|^2A(X,Y)$, $\nabla^2 u(X,\nu)=X(|\nabla u|)$.
\end{lma}

\begin{lma}\label{l-basic-2} Let $\Sigma_\tau$ be as in the previous lemma. Then
\bee
\frac{d}{d\tau}\int_{\Sigma_\tau}|\nabla u|H= \int_{\Sigma_\tau}  \lf( \frac12(\cS^\tau- \frac n{n-1}H^2) +H\frac{\Delta u}{|\nabla u|}\ri) -E(\tau)
\eee
where
\bee
E(\tau)=\int_{\Sigma_\tau} \lf(\frac{|\wt\nabla |\nabla u|^2|}{|\nabla u|^2} +\frac12(\cS+|\mathring A|^2)\ri).
\eee
and
\bee
\frac{d}{d\tau}\int_{\Sigma_\tau}|\nabla u|^2 =\int_{\Sigma_\tau}\lf(2 \Delta u-|\nabla u|H\ri).
\eee
\end{lma}

\section{On equation \eqref{eq-bdry-system}}\label{s-equations}
In $n$-dimension, \eqref{eq-bdry-system} takes the form:
\be\label{e-equations}
\Delta u=\a u_{\nu\nu}+\frac{(n-1)|\nabla u|^2}u.
\ee
where $2-n\le \a\le 1.$  We  assume $u>0$ and $|\nabla u|>0$.

\begin{lma}\label{l-equations}
\begin{enumerate}
  \item [(i)] If $2-n<\a<1$, then $u$ is a solution to \eqref{e-equations} if and only if $U=u^{-a}$ is a $p$-harmonic function. Here $p=2-\a$ so that $1<p<n$ and $a=(n-p)/(p-1)$.
  \item[(ii)] If $\a=1$, then $u$ is a solution to    \eqref{e-equations} if and only if the level sets of  $U=(n-1)\log u $ is a solution to the inverse mean curvature flow. Namely
      $$\div(\frac{\nabla U}{|\nabla U|})=|\nabla U|.
      $$
  \item[(iii)]  If $\a=2-n$, then $u$ is a solution to    \eqref{e-equations} if and only $U= \log u $ is $n$-harmonic.
\end{enumerate}

\end{lma}
\begin{proof} This follows from direct computations. We only prove (i).  Let $U$ be $p$-harmonic in $(M^n,g)$ with $1<p<n$, so that
\bee
\Delta U=(2-p)U_{\nu\nu},
\eee
with $\nu=\nabla U/|\nabla U|$. Here we assume that $U>0, |\nabla U|>0$.
Let $u=U^{-\frac1a}$ where $a=(n-p)/(p-1)$. Then
\bee
\begin{split}
\Delta u=&-\frac1a U^{-1-\frac1a}\Delta U+\frac1a(1+\frac1a)U^{-2-\frac1a}|\nabla U|^2\\
=&-\frac1a U^{-1-\frac1a}(2-p)U_{\nu\nu}+\frac1a(1+\frac1a)U^{-2-\frac1a}|\nabla U|^2
\end{split}
\eee
Since $\wt\nu=:\nabla u/|\nabla u|=-\nu$,
\bee
u_{\nu\nu}=-\frac1a U^{-1-\frac1a}U_{\nu\nu}+\frac1a(1+\frac1a)U^{-2-\frac1a}|\nabla U|^2.
\eee
Hence
\bee
\begin{split}
\Delta u-(2-p)u_{\nu\nu}=&(p-1)\frac1a(1+\frac1a)U^{-2-\frac1a}|\nabla U|^2\\
=&(n-1)a^{-2}U^{-2-\frac1a}|\nabla U|^2\\
=&(n-1)\frac{|\nabla u|^2}u,
\end{split}
\eee
because $|\nabla u|^2=a^{-2}U^{-2-\frac2a}|\nabla U|^2$.
Hence $U$ is $p$-harmonic. The converse can be proved similarly.
\end{proof}

 \begin{lma}\label{c-integral formula-h-more}
Suppose the equality holds in Corollary \ref{c-integral formula-h} and $ \alpha  = 1 $.
Then $g$ can be written as
\bee
\begin{split}
g =&\frac{1}{\frac12\chi-  m_H\rho^{-1}}d\rho^2+\rho^2\gamma_0
\end{split}
\eee
where
$$
m_H=\lf(\frac{|\p_-M|}{4\pi}\ri)^\frac12\lf(2\pi\chi-\frac14\int_{\p_-M}H^2 \ri),
$$
 and $\gamma_0$ is a metric of constant curvature and with area $4\pi$. Moreover $u=C\rho$ for some constant $C$.
\end{lma}

\begin{proof}
In the proof of Corollary \ref{c-integral formula-h}, if $\a=1$, then $\b=0$ and by   Theorem \ref{T: integral formula-h-intro},
\bee
2\pi \chi \tau+\int_{\{u=\tau\}}\lf(|\nabla u|^2- H|\nabla u|\ri)
=2\pi   \chi c_-+\int_{\p_-M}\lf(|\nabla u|^2- H|\nabla u|\ri)
\eee
for all $c_-\le \tau\le c_+$. In terms of $v$, we have
\bee
t^\frac12\lf(2\pi \chi -\frac14  \int_{\{v=t\}}H^2\ri)
=t^\frac12_0\lf(2\pi \chi -\frac14  \int_{\{v=t_0\}}H^2\ri)=:m
\eee
because $H=\frac{|\nabla v|}v=\frac{\eta}t$. We have
\bee
\eta^2(t)=4t_0|\Sigma_{t_0}|^{-1}t \lf(2\pi\chi-mt^{-\frac12}\ri)
\eee
because $|\Sigma_t|=tt_0^{-1}|\Sigma_{t_0}|$.
Hence, if we let $t=r^2$, then
\bee
\begin{split}
g=&\frac{|\Sigma_{t_0}|dt^2}{4t_0 \lf(2\pi\chi-mt^{-\frac12}\ri)}+tt_0^{-1}\gamma_{t_0}\\
=&\frac{|\Sigma_{t_0}|dr^2}{ t_0 \lf(2\pi\chi-mr^{-1}\ri)}+r^2t_0^{-1}\gamma_{t_0}\\
=&\frac{d\rho^2}{\frac12\chi-\wt m\rho^{-1}}+\rho^2\gamma_0 ,
\end{split}
\eee
where $\rho=\lf(\frac{|\Sigma_{t_0}|}{4\pi t_0}\ri)^\frac12 r$ and
$\wt m=\lf(\frac{|\Sigma_{t_0}|}{4\pi t_0}\ri)^\frac12 m$. $\gamma_0$ is a metric of constant curvature and with area $4\pi$.
\end{proof}

\section{Comparing $\cA(t), \cB(t), F(t)$ to the Hawking mass}\label{s-F}
We compare $\cA(t), \cB(t), F(t)$ with the Hawking mass. Let  $1>u>0$ be $p$-harmonic with $1<p<3$ with  $|\nabla u|>0$.
Recall that
$$
   \cB(t)=4\pi t-(ca)^{-2}t^{2a+1}\int_{\{ u = 1-ct^{-a}\}   }|\nabla u|^2.
$$
   Denote $c$ by $c_p$ because it depends on $p$. Let    $U=(1-p)\log (1-u)$. Direct computations show that   $U$ satisfies  (see \cite{Moser2007}):
   $$
   \div\lf(|\nabla U|^{p-2}\nabla U\ri)=|\nabla U|^p.
   $$
   In terms of $U$,
\bee
   \cB(t)=4\pi c_p^\frac 1a\exp\lf(\frac{\tau}{3-p}\ri)\lf(1  -\frac 1{4\pi(3-p)^2}\int_{\{ U=\tau  \}}|\nabla U|^2\ri)
\eee
    where $\tau=(1-p)\log (c_pt^{-a})$. By \cite{Moser2007}, see also \cite{KotschwarNi2009}, $U$ will converge to the weak solution $U_1$ of the inverse mean curvature flow as $p\to 1$ in some weak sense. If the convergence is $C^2$,  then one can see that $\cB$ will converge to a constant multiple of the Hawking mass of a level set of $U_1$. Similarly, $\cA$ and $F$ also converge to constant multiples of the Hawking mass.

On the other hand,
\bee
\begin{split}
| \nabla U |^p  =&\div(\frac{\nabla U}{|\nabla U|}|\nabla U|^{p-1})\\
=&|\nabla U|^{p-2}\lf( | \nabla U | H+(p-1)U_{\nu\nu}\ri).
\end{split}
\eee
Here $H$ is the mean curvature of the level surfaces of $U$. Hence,
\be
| \nabla U|^2 = |\nabla U|H+(p-1)U_{\nu\nu} .
\ee
This gives the following forms of $\cB, \cA$ in terms of the mean curvature, which also indicate that as $p\to 1$,  $\cB, \cA$ will converge to multiples of the Hawking mass of the level surface of the inverse mean curvature flow.
\begin{prop}\label{p-B-Hawking}
\bee
\begin{split}
\cB(\tau)
=&4\pi c_p^\frac 1a\exp\lf(\frac{\tau}{3-p}\ri)\lf[1-\frac1{4\pi(3-p)^2} \int_{\{U=\tau\}}  \lf(    H+(p-1)\frac{U_{\nu\nu}}{|\nabla U|} \ri)^2  \ri] ,
\end{split}
\eee
\bee
\begin{split}
\cA(\tau)
  =& 8\pi c_p^\frac 1a \exp( \frac \tau{3-p})\lf[1-\frac 1{8\pi(3-p)} \int_{\{U=\tau\}}    H \lf(    H+(p-1)\frac{U_{\nu\nu}}{|\nabla U|} \ri)\ri].
\end{split}
\eee
\end{prop}

\end{document}